\documentclass[11pt]{article}
\usepackage[a4paper,margin=1in]{geometry}
\usepackage[T1]{fontenc}
\usepackage[utf8]{inputenc}
\usepackage{amsmath,amssymb,amsfonts,amsthm}
\usepackage{mathrsfs}
\usepackage{enumitem}
\usepackage{xcolor}

\usepackage[colorlinks=true,linkcolor=blue,citecolor=red,urlcolor=blue]{hyperref}
\usepackage{microtype}

\newtheorem{theorem}{Theorem}[section]
\newtheorem{lemma}[theorem]{Lemma}
\newtheorem{proposition}[theorem]{Proposition}

\newtheorem*{remark*}{Remark}

\allowdisplaybreaks

\title{Radial Nodal Dirichlet Solutions of Singular Elliptic Equations: Global Branches, Endpoint Asymptotics, and Morse Indices}

\date{}
\author{Wenjing Chen\footnote{E-mail address:\, {\tt wjchen@swu.edu.cn}.}\  \ and Zexi Wang\footnote{Corresponding author. E-mail address:\, {\tt zxwangmath@163.com}.} \\
\footnotesize  School of Mathematics and Statistics, Southwest University,
Chongqing, 400715, P.R. China}

\begin{document}
\maketitle

\begin{abstract}
{We establish sharp finite-ball shooting classifications for radial nodal
Dirichlet solutions of the logarithmic equation and, within the maximal
simple-zero shooting class, of the sublinear scalar-field equation.  Recent
whole-space uniqueness and phase-transition theorems are used as external
inputs, while the finite-ball zero-curve ranges, endpoint asymptotics, and
spectral consequences are proved here.  The two models require different
singular analyses: in the logarithmic problem the nonlinearity is not locally
Lipschitz at a nodal zero and the linearized potential diverges there, whereas
in the sublinear problem the limiting whole-space profile reaches a double
zero at a finite support radius, beyond which continuation is nonunique.}

For every $R>0$ and $k\ge0$, the logarithmic problem has, up to sign, a unique
radial Dirichlet solution with exactly $k$ interior zeros. Its shooting
height $\beta_k^{\log}(R)$ is a strictly decreasing $C^1$ bijection from
$(0,\infty)$ onto $(\alpha_k^{\log},\infty)$ and satisfies
\[
 \beta_k^{\log}(R)\longrightarrow\alpha_k^{\log}
 \quad(R\to\infty),\qquad
 \log\bigl(\beta_k^{\log}(R)^2\bigr)
 =\frac{\rho_{k+1}^2}{R^2}+\kappa_{k+1,n}+o(1)
 \quad(R\downarrow0),
\]
{where $\kappa_{k+1,n}>0$ is given by an explicit Bessel integral.}
For the sublinear problem, the maximal simple-zero branch exists precisely for
$R\in(\rho_{k+1},S_k)$, where $\rho_{k+1}$ is the $(k+1)$-st positive zero of the regular Bessel profile $\Phi_n$, and $S_k$ is the support radius of the unique
compactly supported $k$-node whole-space bound state. Its shooting height is
strictly decreasing and obeys
\[
 \beta_k^{\rm sub}(R)\longrightarrow\alpha_k^{\rm sub}
 \quad(R\uparrow S_k),\qquad
 (R-\rho_{k+1})\bigl(\beta_k^{\rm sub}(R)\bigr)^{2-q}
 \longrightarrow \mathfrak c_{k+1,q,n}
 \quad(R\downarrow\rho_{k+1}),
\]
with an explicit constant $\mathfrak c_{k+1,q,n}>0$. {The zero extensions converge to the compactly supported limiting state in
$W^{1,\infty}(0,S_k)$.  We also record its free-boundary profile and the
corresponding first- and second-derivative rates.  After
rescaling to the unit ball, both families form strictly monotone no-fold
parameter branches on their maximal parameter intervals.}

{Finally, we develop a radial Sturm theory for the singular linearized
potentials by approximation with bounded potentials and convergence of the
associated forms.  For the power benchmark and for the logarithmic and
sublinear branches above, the stopped $k$-node profile is radially
nondegenerate and has radial Morse index $k+1$; the degree-one angular sector
has exactly $k$ negative eigenvalues.  Spherical-harmonic decomposition yields
a finite formula for the full Morse index and shows that possible nonradial
degeneracy can occur only in finitely many angular sectors of degree at least
two.}
\end{abstract}



\section{Introduction and main results}
Let $B_R\subset\mathbb R^n$ be the ball of radius $R>0$ centered at the
origin.  We study regular radial solutions of
\begin{equation}\label{eq:unified-ball}
 \Delta U+f(U)=0\quad\text{in }B_R,
 \qquad U=0\quad\text{on }\partial B_R,
\end{equation}
that are positive at the origin and have a prescribed number of simple zeros
in $(0,R)$.  A regular radial solution is represented by a function
$U\in C^1([0,R])\cap C^2((0,R))$ satisfying $U'(0)=0$ and the radial
equation pointwise; continuity of the nonlinearities makes the equation
classical also at a nodal zero.  The nonlinearities considered are
\begin{align}
 f_{\rm pow}(u)&=-u+|u|^{p-1}u,
 &1<p<p_{\rm S},                                      \label{eq:intro-power}\\
 f_{\log}(u)&=u\log u^2,                                \\
 f_{\rm sub}(u)&=u-|u|^{q-2}u,
 &1<q<2,                                                  \label{eq:intro-sub}
\end{align}
where
\begin{equation*}
 p_{\rm S}=\begin{cases}
 \infty,\qquad  &n=2,\\
 \dfrac{n+2}{n-2},\qquad &n\ge3.
 \end{cases}
\end{equation*}
The logarithmic and sublinear equations are the {principal finite-ball
problems studied in this paper}. The power equation is retained only as an
already classified {benchmark for the shooting and spectral mechanisms}.
%\rev{Throughout the paper, $k$ denotes the number of interior nodal zeros of
%the stopped radial profile.  Thus $k=0$ is included and denotes the positive
%(ground-state) branch; statements specifically about sign-changing
%whole-space bound states are naturally indexed by $k\ge1$.}

%\medskip

The whole-space scalar-field problem
 \begin{equation}\label{scp}
 \Delta u+f(u)=0\quad\hbox{in }\mathbb R^n,
 \qquad u(x)\to0\quad\hbox{as }|x|\to\infty,
 \end{equation}
was studied systematically in the foundational work of Berestycki and Lions
\cite{BL1,BL2}.  {Under natural structural assumptions on the nonlinearity $f$, they established
the existence of a ground state and infinitely many bound states. They also
conjectured uniqueness of a bound state with a prescribed number of nodes in
several important model classes. Here, a nontrivial decaying solution of
\eqref{scp} is called a bound state; its positive least-energy member is the
ground state, while a sign-changing radial member is a nodal bound state. The
radial shooting approach has a longer history. Early
nonlinear eigenvalue and oscillation arguments can be traced back to Kolodner
\cite{Kolodner}, while the prescribed-zero existence theory was developed by}
Jones--K\"upper \cite{JonesKupper} and McLeod--Troy--Weissler
\cite{MTW}.
The uniqueness of the ground state was subsequently advanced through the works of
Chen--Lin \cite{ChenLin}, Coffman \cite{Coffman,Coffman96}, Kwong
\cite{Kwong,KZ91},  McLeod--Serrin \cite{McLeodSerrin},  Peletier--Serrin \cite{PeletierSerrin},  Pucci--Serrin \cite{PucciSerrin98}, Serrin--Tang
\cite{SerrinTang}, and Yanagida \cite{Y91}.

{Three recent classification results provide the whole-space
phase-transition input used below. Tang \cite{Tang} proved the complete
prescribed-node classification for the power equation in dimensions $n\ge3$;
Liu--Sun--Zou \cite{LiuSunZou} proved uniqueness of logarithmic bound states
with every prescribed nodal number for $n\ge2$ and established the variation
interlacing needed here for $n\ge3$; and Zhang--Zhang \cite{ZhangZhang}
proved the corresponding uniqueness and phase-transition theory for the
compactly supported sublinear bound states, together with the planar extension
of the power classification.  These whole-space theorems are used as inputs;
the exact finite-ball ranges, endpoint laws, and singular operator arguments
below are not consequences of whole-space uniqueness alone.}

For the power nonlinearity \eqref{eq:intro-power}, prescribed-node
uniqueness was proved by Cort\'{a}zar, Garc\'{i}a-Huidobro and Yarur
\cite{CGY11} in dimensions $n=2,3,4$ for restricted exponent ranges, by
Ao--Wei--Yao \cite{AWY16} for the one-node state when the exponent is close to
the critical exponent, and by Cohen--Li--Schlag \cite{CLS24}, using
computer-assisted methods, for the first twenty radial excited states of the
cubic equation in dimension three.
A major breakthrough was subsequently achieved by Tang \cite{Tang},
who obtained the complete classification of bound states in all dimensions $n\geq3$.
{More precisely, for $k\ge0$, let $\alpha_k^{\rm pow}$ be the initial
height of the unique whole-space $k$-node bound state, and let
$z_{k+1}^{\rm pow}(\alpha)$ denote the $(k+1)$-st positive zero of
the shooting solution; for $k=0$, $\alpha_0^{\rm pow}$ denotes the
ground-state height. The proof of Tang's Theorems~1 and~2, together
with the classical ground-state case, gives}
\begin{equation*}
 z_{k+1}^{\rm pow}:(\alpha_k^{\rm pow},\infty)
 \longrightarrow(0,\infty)
 \quad\hbox{is strictly decreasing},
\end{equation*}
with
\begin{equation*}
 \lim_{\alpha\downarrow\alpha_k^{\rm pow}}
 z_{k+1}^{\rm pow}(\alpha)=\infty,
 \qquad
 \lim_{\alpha\to\infty}z_{k+1}^{\rm pow}(\alpha)=0.
\end{equation*}
Consequently, for every $R>0$, there is a unique shooting height
$\beta_k^{\rm pow}(R)>\alpha_k^{\rm pow}$ satisfying
\[
 z_{k+1}^{\rm pow}(\beta_k^{\rm pow}(R))=R.
\]
{Equivalently, every ball contains, up to sign, exactly one
radial solution with $k$ interior zeros.}
Moreover,
\begin{equation}\label{eq:intro-power-beta}
 \beta_k^{\rm pow}:(0,\infty)
 \longrightarrow(\alpha_k^{\rm pow},\infty)
 \quad\hbox{is $C^1$ and strictly decreasing},
\end{equation}
and
\begin{equation}\label{eq:intro-power-beta-endpoints}
\lim\limits_{R\downarrow0} \beta_k^{\rm pow}(R)=\infty,
 \qquad
\lim\limits_{R\to\infty} \beta_k^{\rm pow}(R)=\alpha_k^{\rm pow}.
\end{equation}
Thus Tang's whole-space classification theorem, together with its zero-curve
argument, already implies the existence and uniqueness result on finite balls, as well as the monotonicity of the shooting height in the power case.
For $n=2$, Zhang--Zhang's phase-transition argument yields the corresponding
conclusion for every $p>1$; see \cite[Section~6]{ZhangZhang}. Earlier
finite-ball uniqueness criteria for superlinear and sublinear nodal profiles
were obtained by Tanaka \cite{TanakaSuper,TanakaSub,Tanaka}. Their structural and
regularity assumptions are not designed for the present logarithmic
linearization or for a nonlinearity with nonunique continuation from a double
zero. Variational nodal solutions for broad sublinear-type Dirichlet
problems were studied in \cite{BonheureSantosPariniTavaresWeth}; those results
address existence and symmetry rather than the maximal radial shooting
branches considered here.

{Thus the power theory is used here as an input and a benchmark, rather
than as a new finite-ball classification result.  The issue is which parts of
the inverse zero-curve picture
\eqref{eq:intro-power-beta}--\eqref{eq:intro-power-beta-endpoints} persist for
the two singular models.  This requires Cauchy theory through logarithmic or
sublinear singular zeros, a separate treatment of the compact-support
endpoint in the sublinear problem, and a singular Sturm argument for the
linearized operators.  The precise finite-ball conclusions are summarized
after the two model-specific backgrounds below.}

The logarithmic nonlinearity originates in the nonlinear wave-mechanics model introduced by Bialynicki-Birula and Mycielski \cite{BBM}, and its variational framework is closely connected with the logarithmic Sobolev inequality \cite{Gross,LL01}.
The existence and variational properties of stationary logarithmic equations
have been studied extensively; see, for instance, \cite{AJ20,LWZ26,S19,SquassinaSzulkin,TanakaZhang,WangZhang,WangZhangZhang}. In the whole space, the positive solution is given explicitly by the Gausson
\begin{equation}\label{eq:intro-gausson}
u_0(x)=\exp\left(\frac n2-\frac{|x|^2}{2}\right).
\end{equation}
Its uniqueness was established by Troy \cite{Troy} for $2\leq n\leq 9$ and by
d'Avenia--Montefusco--Squassina \cite{DMS} for $n\geq3$ using different methods.
More recently, Liu--Sun--Zou \cite{LiuSunZou} proved the uniqueness of
{radial logarithmic bound states with every prescribed nodal number} in
all dimensions $n\ge2$. We denote the associated shooting heights by
\begin{equation}\label{eq:alphak}
e^{\frac{n}{2}}=\alpha_0^{\log}
<\alpha_1^{\log}<\alpha_2^{\log}<\cdots,
\qquad \alpha_k^{\log}\to\infty.
\end{equation}
{The logarithmic problem on a finite ball has two singular
features.  The map $u\mapsto u\log u^2$ is not locally Lipschitz at $u=0$,
and the linearized coefficient $\log u^2+2$ diverges at every nodal zero.
Near a simple zero the divergence is logarithmic and therefore locally
integrable.  We exploit this fact to prove uniqueness and continuous
dependence for both the nonlinear shooting equation and its variation across
simple zeros; this supplies the differentiability of the zero curves that is
needed for the finite-ball classification.}

For the sublinear model
\[
 -\Delta u-u+|u|^{q-2}u=0,
 \qquad 1<q<2,
\]
whole-space ground states and bound states are compactly supported.
The symmetry and free-boundary structure of ground states were studied in
\cite{CortazarDelPinoElgueta1,CortazarDelPinoElgueta2,CortazarElguetaFelmer,Gui,IkomaTanakaWangZhang},
while the existence of bound states and the qualitative asymptotic behavior of their supports were investigated in
\cite{BalabaneDolbeaultOunaies,GazzolaSerrinTang}.
This compact-support phenomenon is a consequence of the non-Lipschitz
behavior of the nonlinearity $f_{\rm sub}(u)=u-|u|^{q-2}u$ at $u=0$; see, for instance, Pucci--Serrin \cite{PucciSerrin}. Despite the singular nature of $f_{\rm sub}$, the uniqueness of the ground state was established in \cite{PeletierSerrin,SerrinTang}. More recently,
Zhang--Zhang \cite{ZhangZhang} proved the uniqueness of compactly supported
{radial bound states with every prescribed nodal number} in every
dimension $n\ge2$.

{Unlike the power and logarithmic profiles, a sublinear bound
state reaches a double zero at a finite support radius.  Beyond that point the
nonlinear Cauchy problem is not uniquely solvable.  Consequently, the
finite-ball classification must be formulated in the class of shooting
profiles that encounter only simple zeros before the boundary.  In this class
the relevant zero curve has the exact range $(\rho_{k+1},S_k)$: the lower
endpoint is selected by a linear Bessel limit at large shooting height, while
the upper endpoint is the support radius of the compactly supported
whole-space {bound state}.}

{The finite-ball contributions of this paper can now be stated without
including the already classified power case.  For the logarithmic equation,
we determine the exact range and transversality of each relevant stopped zero
curve and obtain a second-order small-radius expansion.  For the sublinear
equation, we determine the maximal simple-zero radius interval, the sharp
Bessel law at its lower endpoint, and $W^{1,\infty}$ convergence to the
compactly supported state at its upper endpoint; we also record the known
leading free-boundary profile and its equivalent first-derivative law, and
derive the second-derivative rate used in the spectral discussion.  For both
singular models, we then connect zero-curve
transversality with the singular linearized spectrum, proving radial
nondegeneracy and the exact radial Morse index.  Finally, for the power
benchmark and both singular models, spherical-harmonic decomposition gives
the degree-one count, bounds all higher angular counts, and yields the full
Morse-index formula stated below.}

For all three models, the radial shooting problem takes the form
\begin{equation}
 \begin{cases}
 u''+\dfrac{n-1}{r}u'+f(u)=0,\qquad  r>0,\\[4pt]
 u(0)=\alpha>0,\qquad u'(0)=0.
 \end{cases}
 \label{eq:intro-shooting-unified}
\end{equation}
Whenever it is defined, let $z_j(\alpha)$ denote the $j$-th positive simple
zero of $u(\cdot,\alpha)$ and let
\[
 v(\cdot,\alpha):=\partial_\alpha u(\cdot,\alpha)
\]
denote the shooting variation.  The key phase-transition information is
the strict sign relation at each zero
\begin{equation*}
 u_r(z_j(\alpha),\alpha)v(z_j(\alpha),\alpha)>0.
\end{equation*}
Differentiating the identity $u(z_j(\alpha),\alpha)=0$ with respect to $\alpha$ yields
\begin{equation}\label{eq:intro-zero-derivative}
 z_j'(\alpha)
 =-\frac{v(z_j(\alpha),\alpha)}
         {u_r(z_j(\alpha),\alpha)}<0.
\end{equation}
Thus the finite-ball classification reduces to determining the exact
range of the relevant zero curve. Its inverse gives precisely the
shooting height as a function of the ball radius, and
\eqref{eq:intro-zero-derivative} implies that this shooting height is strictly monotone.

The same shooting variation also controls the spectrum. Let $U$ be a stopped
radial profile in $B_R$ and define
\[
 L_U:=-\Delta-f'(U).
\]
On the radial space $H^1_{0,{\rm rad}}(B_R)$, consider the quadratic
form
\begin{equation}\label{eq:intro-radial-form}
 \mathcal Q_U(\phi)
 :=\int_0^R\left(|\phi'(r)|^2-f'(U(r))\phi^2(r)\right)
 r^{n-1}\,dr.
\end{equation}
In the logarithmic and sublinear cases, this is understood as the
closed, lower-bounded radial quadratic form associated with the
locally integrable singular potential $f'(U)$; its construction is recalled in
Section~\ref{sec:proof-morse}.  The \emph{radial Morse index} of $U$ is defined by
\begin{equation*}
\begin{aligned}
 m_{\rm rad}(U)
 :=\sup\Bigl\{\dim X:X\subset H^1_{0,{\rm rad}}(B_R),\
 \mathcal Q_U(\phi)<0\ \hbox{for every }0\ne\phi\in X\Bigr\}.
\end{aligned}
\end{equation*}
See, for instance, \cite{DIP17,DIP17',HRS11}.
Equivalently, $m_{\rm rad}(U)$ is the number of negative eigenvalues, counted with
multiplicity, of the radial self-adjoint realization of $L_U$. We denote its radial Dirichlet kernel by $\ker_{\rm rad}L_U$ and call $U$ \emph{radially nondegenerate} if $\ker_{\rm rad}L_U=\{0\}$. The shooting
variation $v$ is the regular zero-energy solution of the radial
linearized equation. Consequently, the Sturm oscillation theorem
identifies ${m_{\rm rad}}(U)$ with the number of zeros of $v$ in
$(0,R)$, while the condition $v(R)\ne0$ rules out a radial zero
eigenvalue and hence yields radial nondegeneracy.

Let $\Phi=\Phi_n$ be the regular solution of
\begin{equation}\label{eq:main-bessel}
\Phi''+\frac{n-1}{s}\Phi'+\Phi=0,
\qquad \Phi(0)=1,
\qquad \Phi'(0)=0,
\end{equation}
and denote its positive zeros by
$
0<\rho_1<\rho_2<\cdots.
$
For $j\ge1$, set
\begin{equation}\label{eq:main-log-kappa}
 \kappa_{j,n}:=
 -\frac{2\displaystyle\int_0^{\rho_j}
 s^{n-1}\Phi_n^2(s)\log\Phi_n^2(s)\,ds}
 {\rho_j^n\Phi_n'^2(\rho_j)}>0.
\end{equation}
The integrand is defined to be zero at the zeros of $\Phi_n$.

\begin{theorem}\label{thm:log-main}
Let $n\ge2$, $R>0$, and $k\ge0$. The Dirichlet problem
\begin{equation}\label{eq:main-log-ball}
 \begin{cases}
 \Delta U+U\log U^2=0,&x\in B_R,\\
 U=0,&x\in\partial B_R,
 \end{cases}
\end{equation}
where $U\log U^2$ is extended continuously by zero at $U=0$, admits a
unique regular radial solution $U^{\log}_{k,R}$ satisfying
$U^{\log}_{k,R}(0)>0$ and having exactly $k$ zeros in $(0,R)$. All its
interior zeros and its boundary zero are simple. Every nontrivial radial
solution of \eqref{eq:main-log-ball} with exactly $k$ interior zeros is equal
to either $U^{\log}_{k,R}$ or $-U^{\log}_{k,R}$.

Define $\beta_k^{\log}(R):=U^{\log}_{k,R}(0)$. Then
\[
 \beta_k^{\log}:(0,\infty)\longrightarrow
 (\alpha_k^{\log},\infty)
\]
is a {strictly decreasing $C^1$ bijection}. With
$u(\cdot,\alpha)$ denoting the shooting solution and
$v=\partial_\alpha u$ its variation, one has
\begin{equation*}
 \bigl(\beta_k^{\log}\bigr)'(R)
 =-\frac{u_r(R,\beta_k^{\log}(R))}
 {v(R,\beta_k^{\log}(R))}<0.
\end{equation*}
Moreover,
\begin{equation}\label{eq:main-log-limits}
 \lim_{R\to\infty}\beta_k^{\log}(R)=\alpha_k^{\log},
 \qquad
 \log\bigl(\beta_k^{\log}(R)^2\bigr)
 =\frac{\rho_{k+1}^2}{R^2}+\kappa_{k+1,n}+o(1)
 \quad(R\downarrow0).
\end{equation}
Here $\alpha_k^{\log}$ is the $k$-th logarithmic whole-space shooting
height defined in \eqref{eq:alphak}, and $\rho_{k+1}$ is the
$(k+1)$-st positive zero of $\Phi_n$ in \eqref{eq:main-bessel}.
Furthermore, as $R\to\infty$,
\begin{equation}\label{eq:main-log-convergence}
 U^{\log}_{k,R}\longrightarrow u_k^{\log}
 \quad\text{in }C^1_{\mathrm{loc}}(\mathbb R^n)
 \quad\text{and in }C^2_{\mathrm{loc}}(\mathbb R^n\setminus Z_k),
\end{equation}
where $u_k^{\log}$ is the unique whole-space $k$-node logarithmic bound
state and $Z_k$ is its nodal set.
\end{theorem}

\begin{remark*}
{\rm The two-dimensional case requires separate attention. Tang's original
argument \cite{Tang} for the power nonlinearity relies on the assumption $n\geq3$
in establishing the positivity of an auxiliary quantity.
Zhang--Zhang \cite{ZhangZhang} observed that, when $n=2$, the relevant bridge quantity $Q_n$
reduces to $Q_2$, whose positivity is already available from the phase
induction. The same modification applies to the logarithmic problem
and completes the phase-transition argument in dimension two. {We
give the details of the planar logarithmic argument in
Proposition~\ref{prop:log-planar}, because they are needed to establish
Theorem~\ref{thm:log-whole} for $n=2$.}}
\end{remark*}

For the next theorem, define
\begin{equation}\label{eq:sub-constant-intro}
 \mathfrak c_{j,q,n}
 :=\frac{\displaystyle\int_0^{\rho_j}
 s^{n-1}|\Phi_n(s)|^q\,ds}
 {\rho_j^{n-1}\Phi_n'^2(\rho_j)}>0.
\end{equation}

\begin{theorem}\label{thm:sub-main}
Let $n\ge2$, $1<q<2$, and $k\ge0$. Let $\alpha_k^{\rm sub}$ be the
shooting height of the unique compactly supported $k$-node sublinear bound
state $u_k^{\rm sub}$, and let $S_k$ be its support radius. Then
$S_k>\rho_{k+1}$.

A radial solution $U$ of
\begin{equation}\label{eq:main-sub-ball}
 \begin{cases}
 \Delta U+U-|U|^{q-2}U=0,&x\in B_R,\\
 U=0,&x\in\partial B_R,
 \end{cases}
\end{equation}
with $U(0)>0$, whose zero set in $(0,R)$ consists of exactly $k$ simple
zeros and whose boundary zero is simple, exists if and only if
\[
 R\in(\rho_{k+1},S_k).
\]
For every such $R$, this solution is unique and is denoted by
$U^{\rm sub}_{k,R}$. Every nontrivial radial solution in the same
simple-zero class is either $U^{\rm sub}_{k,R}$ or
$-U^{\rm sub}_{k,R}$.

The shooting height $\beta_k^{\rm sub}(R):=U^{\rm sub}_{k,R}(0)$ is
$C^1$ and strictly decreasing on $(\rho_{k+1},S_k)$. With
$u(\cdot,\alpha)$ denoting the shooting solution and
$v=\partial_\alpha u$ its variation, one has
\begin{equation*}
 \bigl(\beta_k^{\rm sub}\bigr)'(R)
 =-\frac{u_r(R,\beta_k^{\rm sub}(R))}
 {v(R,\beta_k^{\rm sub}(R))}<0.
\end{equation*}
Moreover,
\begin{equation}\label{eq:main-sub-limits}
 \lim_{R\uparrow S_k}\beta_k^{\rm sub}(R)=\alpha_k^{\rm sub},
 \qquad
 \lim_{R\downarrow\rho_{k+1}}
 (R-\rho_{k+1})\bigl(\beta_k^{\rm sub}(R)\bigr)^{2-q}
 =\mathfrak c_{k+1,q,n}.
\end{equation}
If $\widetilde U^{\rm sub}_{k,R}$ denotes the zero extension of the
{radial profile} $U^{\rm sub}_{k,R}$ from {$[0,R]$ to $[0,S_k]$}, then
\begin{equation}\label{eq:main-sub-convergence}
 \widetilde U^{\rm sub}_{k,R}\longrightarrow u_k^{\rm sub}
 \quad\text{in }W^{1,\infty}(0,S_k)
 \quad\text{as }R\uparrow S_k.
\end{equation}
In particular, the convergence holds {in $C^1_{\rm loc}([0,S_k))$ and in
$C^2_{\rm loc}$ away from the limiting nodal zeros}.
The zero extension is used only to compare functions on the fixed interval
$[0,S_k]$; it is not asserted to solve the equation classically across the
simple zero $r=R$.

Let $\sigma_k=(-1)^k$ and
\begin{equation*}
 C_q:=\left(\frac{(2-q)^2}{2q}\right)^{\frac{1}{2-q}}.
\end{equation*}
The terminal nodal domain of the limiting bound state has the exact
free-boundary expansion
\begin{align}
 \sigma_k u_k^{\rm sub}(r)
 &=C_q(S_k-r)^{\frac{2}{2-q}}(1+o(1)),\label{eq:main-free-u}\\
 -\sigma_k (u_k^{\rm sub})'(r)
 &=\frac{2C_q}{2-q}(S_k-r)^{\frac{q}{2-q}}(1+o(1)),\label{eq:main-free-du}\\
 \sigma_k (u_k^{\rm sub})''(r)
 &=C_q^{q-1}(S_k-r)^{\frac{2(q-1)}{2-q}}(1+o(1)),\label{eq:main-free-ddu}
\end{align}
as $r\uparrow S_k$.
\end{theorem}

\begin{remark*}
{\rm The restriction to the simple-zero class in
Theorem~\ref{thm:sub-main} is essential. Once a shooting profile reaches a
double zero, the sublinear Cauchy problem generally admits nonunique
continuations beyond that point. Such continuations are excluded from the
maximal simple-zero branch. The theorem neither classifies nor asserts
uniqueness among profiles with dead cores, waiting intervals, or repeated
departures from a double zero.}
\end{remark*}

\begin{theorem}\label{thm:branch-main}
Let $n\ge2$ and $k\ge0$. Put $\lambda=R^2$ and $W(y)=U(Ry)$. Then $W$ solves
\begin{equation}\label{eq:unit-parameter}
 \Delta W+\lambda f(W)=0\quad\hbox{in }B_1,
 \qquad W=0\quad\hbox{on }\partial B_1.
\end{equation}
{The positive-at-the-origin radial solutions described in
Theorems~\ref{thm:log-main} and \ref{thm:sub-main} form the following
strictly monotone no-fold branches on their maximal parameter intervals.}
\begin{enumerate}[label=\textup{(\alph*)}]
\item For $f=f_{\log}$, the central height
\[
 b_k^{\log}(\lambda):=\beta_k^{\log}(\sqrt\lambda)
\]
is a strictly decreasing $C^1$ bijection from $(0,\infty)$ onto
$(\alpha_k^{\log},\infty)$. Its endpoint laws are
\[
 \log\bigl(b_k^{\log}(\lambda)^2\bigr)
 =\frac{\rho_{k+1}^2}{\lambda}+\kappa_{k+1,n}+o(1)
 \quad(\lambda\downarrow0),
 \qquad
 b_k^{\log}(\lambda)\to\alpha_k^{\log}
 \quad(\lambda\to\infty).
\]
\item For $f=f_{\rm sub}$, the maximal simple-zero branch is
parameterized by
\[
 \lambda\in(\rho_{k+1}^2,S_k^2),
 \qquad
 b_k^{\rm sub}(\lambda):=\beta_k^{\rm sub}(\sqrt\lambda).
\]
It is $C^1$ and strictly decreasing, with
\begin{align*}
 (\lambda-\rho_{k+1}^2)
 \bigl(b_k^{\rm sub}(\lambda)\bigr)^{2-q}
 \longrightarrow2\rho_{k+1}\mathfrak c_{k+1,q,n}
 \quad({\lambda\downarrow\rho_{k+1}^2}),\qquad
 b_k^{\rm sub}(\lambda)\longrightarrow\alpha_k^{\rm sub}
 \quad(\lambda\uparrow S_k^2).
\end{align*}
Its closure at $\lambda=S_k^2$ has a double boundary zero and the
free-boundary profile \eqref{eq:main-free-u}--\eqref{eq:main-free-ddu}.
\end{enumerate}
In both cases
\begin{equation}\label{eq:branch-derivative}
 (b_k)'(\lambda)=\frac{\beta_k'(\sqrt\lambda)}{2\sqrt\lambda}<0,
\end{equation}
so neither branch contains a radial shooting fold.
\end{theorem}

\begin{theorem}\label{thm:morse-main}
Let $n\ge2$ and $k\ge0$, and let $U$ be any of the following stopped
profiles.
\begin{enumerate}[label=\textup{(\alph*)}]
\item For $1<p<p_{\rm S}$ and $R>0$, let
$U^{\rm pow}_{k,R}$ be the unique radial power-equation {solution} that is
positive at the origin and has exactly $k$ interior zeros. {Its existence
and uniqueness are exactly the power benchmark recalled above; see Tang
\cite{Tang} for $n\ge3$ and the planar extension of Zhang--Zhang
\cite{ZhangZhang} for $n=2$, together with the classical ground-state case.}
\item For $R>0$, let $U^{\log}_{k,R}$ be the {solution} in
Theorem~\ref{thm:log-main}.
\item For $1<q<2$ and $R\in(\rho_{k+1},S_k)$, let
$U^{\rm sub}_{k,R}$ be the {solution} in Theorem~\ref{thm:sub-main}.
\end{enumerate}
For each $U$ in \textup{(a)}--\textup{(c)}, let
\[
 L_U:=-\Delta-f'(U)
\]
{denote the Dirichlet linearization. In the logarithmic and sublinear
cases, the radial realization is defined by the closed, lower-bounded form
\eqref{eq:intro-radial-form}; the full realization is obtained by the
orthogonal direct sum of the closed spherical-harmonic sector forms described
in Section~\ref{sec:proof-morse}.} Then
\begin{equation*}
 \ker_{\rm rad}L_U=\{0\},
 \qquad
 m_{\rm rad}(U)=k+1.
\end{equation*}
For a spherical harmonic of degree $\ell$, let $N_\ell(U)$ be the number of
{negative Sturm eigenvalues in the radial problem associated with one
fixed degree-$\ell$ spherical harmonic.  The corresponding sector operator
has differential expression}
\begin{equation*}
 -\psi''-\frac{n-1}{r}\psi'
 +\frac{\ell(\ell+n-2)}{r^2}\psi-f'(U)\psi,
 \qquad \psi(R)=0,
\end{equation*}
with the {regular condition $\psi'(0)=0$ when $\ell=0$ and
$\psi(r)=O(r^\ell)$ as $r\downarrow0$ when $\ell\ge1$}. Then
\begin{equation*}
 N_0(U)=k+1,
 \qquad N_1(U)=k,
 \qquad 0\le N_\ell(U)\le k\quad(\ell\ge2),
\end{equation*}
and $N_\ell(U)=0$ for all sufficiently large $\ell$. {If
\[
 d_{0,n}=1,
 \qquad
 d_{\ell,n}=\frac{(2\ell+n-2)(\ell+n-3)!}
 {\ell!(n-2)!}\quad(\ell\ge1)
\]
is the dimension of the degree-$\ell$ spherical harmonics, the full Morse
index is the finite sum
\begin{equation}\label{eq:main-full-morse}
 m(U)=k+1+nk+\sum_{\ell=2}^{\infty}d_{\ell,n}N_\ell(U),
\end{equation}
where $m(U)$ denotes the Morse index of the full Dirichlet realization.}
{The full kernel is trivial in degrees zero and one; hence any nonradial
degeneracy can occur only in one of finitely many degrees $\ell\ge2$.}
\end{theorem}

Sections~\ref{sec:proof-log} and~\ref{sec:proof-sub} establish the two model
branches and their endpoint laws. Section~\ref{sec:parameter-branches}
proves their unified unit-ball formulation. Section~\ref{sec:proof-morse}
develops the singular Sturm argument and the angular decomposition.

\section{The logarithmic Dirichlet problem}
\label{sec:proof-log}

Throughout this section, we write $\alpha_k$ and $\beta_k$ for $\alpha_k^{\log}$ and $\beta_k^{\log}$, respectively.

\subsection{Logarithmic shooting and Cauchy theory}
{In this subsection, we set}
\begin{equation*}
 f(s)=s\log s^2,
 \qquad
 F(s)=\int_0^s f(t)\,dt=\frac12s^2(\log s^2-1),
\end{equation*}
where $f(0)=0$ and $F(0)=0$ are understood by continuity.
Then {$F(\pm e^{1/2})=0$}, $F(s)<0$ for $0<|s|<e^{1/2}$, and $F(s)>0$ for $|s|>e^{1/2}$.
For a radial solution $u$, define the associated energy function by
\begin{equation}\label{eq:energy}
 E(r):=\frac12 u'^2(r)+F(u(r)).
\end{equation}
A direct calculation gives
\begin{equation}\label{eq:Eprime}
 E'(r)=-\frac{n-1}{r}u'^2(r)\leq0.
\end{equation}

We first record a uniqueness and continuous-dependence result for radial linear equations with $L^1$-coefficients. It will be applied near the nodal zeros of $u$, where the coefficient $\log u^2$ is singular but locally integrable.

\begin{lemma}\label{lem:log-L1}
Let $I\subset[0,\infty)$ be a compact interval and let $a\in L^1(I)$. Consider
\begin{equation}\label{eq:linearL1}
 (r^{n-1}y')'+r^{n-1}a(r)y=0\quad{\text{for a.e. }r\in I}.
\end{equation}
If $0\notin I$, then for every fixed $r_0\in I$ and every $(y_0,y_1)\in \mathbb R^2$, there exists a unique function $y\in C^1(I)$ with $r^{n-1}y'\in AC(I)$ satisfying \eqref{eq:linearL1} and $y(r_0)=y_0$, $y'(r_0)=y_1$. If $0\in I$, then for every $y_0\in \mathbb R$, {there exists a unique regular solution} $y\in C^1(I)$ with $r^{n-1}y'\in AC(I)$ satisfying \eqref{eq:linearL1} and $y(0)=y_0$, $y'(0)=0$. Moreover, the solutions depend continuously on both the coefficient and the initial data. More precisely, if $a_m\to a$ in $L^1(I)$ and the corresponding initial data converge, then the associated solutions $y_m$ converge to $y$ in $C^1(I)$.
\end{lemma}

\begin{proof}
Assume first that $0\notin I$. Set $p=r^{n-1}y'$. Then \eqref{eq:linearL1} is equivalent to the first-order system
\[
 y'=r^{1-n}p,
 \qquad
 p'=-r^{n-1}a(r)y.
\]
The coefficient matrix of the first-order system belongs to $L^1(I)$, so existence follows from the associated Volterra integral equation. To prove uniqueness, let two solutions have the same Cauchy data at $r_0$, and denote their difference again by $(y,p)$. For $r\ge r_0$, we have
\[
 |y(r)|+|p(r)|\leq C\int_{r_0}^r(1+|a(s)|)(|y(s)|+|p(s)|)\,ds.
\]
An analogous estimate holds for $r\le r_0$. Gronwall's inequality therefore yields $y\equiv p\equiv0$, and hence uniqueness.
Applying the same argument to the difference of two systems with coefficients $a$ and $\tilde a$ yields continuous dependence of $(y,p)$ in $C(I)\times C(I)$ on both the coefficient and the Cauchy data. Since $I$ is compact and bounded away from the origin, $r^{1-n}$ is bounded on $I$. Recalling that
$y'(r)=r^{1-n}p(r)$,
we conclude that convergence of $(y,p)$ in $C(I)\times C(I)$ implies convergence of $y$ in $C^1(I)$.

It remains to treat the regular initial-value problem at the origin. For a regular solution, the condition $y'(0)=0$ corresponds to $p(0)=0$. Integrating the equation twice, we obtain the Volterra integral equation
\begin{equation}\label{eq:volterra-origin}
 y(r)=y_0-\int_0^r t^{1-n}\int_0^t s^{n-1}a(s)y(s)\,ds\,dt.
\end{equation}
The corresponding integral operator is of Volterra type. The key estimate is
\begin{equation}\label{eq:origin-kernel-estimate}
 \int_0^r t^{1-n}\int_0^t s^{n-1}|a(s)|\,ds\,dt
 \leq C r\int_0^r |a(s)|\,ds.
\end{equation}
Indeed, when $n>2$, changing the order of integration gives
\[
 \int_0^r t^{1-n}\int_0^t s^{n-1}|a(s)|\,ds\,dt=\int_0^r s^{n-1}|a(s)|\int_s^r t^{1-n}\,dt\,ds
 \leq C\int_0^r s|a(s)|\,ds
 \leq C r\int_0^r |a(s)|\,ds.
\]
When $n=2$, the same quantity becomes
\[
 \int_0^r s|a(s)|\log\frac r s\,ds.
\]
 Since $s\log(r/s)\leq C r$ for $0<s<r$, estimate \eqref{eq:origin-kernel-estimate} also holds in dimension two. By \eqref{eq:volterra-origin}, the norm of the integral operator on $C([0,\delta])$ is bounded by
$C\delta\int_0^\delta |a(s)|ds$.
Since $\delta\int_0^\delta |a(s)|ds\rightarrow0$ as $\delta\downarrow0$, we may choose $\delta>0$  sufficiently small so that this bound is less than $1$. The integral operator in \eqref{eq:volterra-origin} is therefore a contraction on $C([0,\delta])$. This yields a unique solution near the origin. The solution can then be continued uniquely by applying the preceding result on compact intervals bounded away from the origin.
 Continuous dependence near the origin follows by applying the same estimates to the difference of the corresponding integral equations. Moreover, $y'(r)=-r^{1-n}\int_0^r s^{n-1}a(s)y(s)\,ds$,
 and hence
 $|y'(r)|\leq\int_0^r |a(s)y(s)|\,ds$. The analogous estimate for the difference of two solutions, together with $a_m\rightarrow a$ in $L^1(I)$ and $y_m\rightarrow y$ in $C(I)$, yields uniform convergence of the derivatives near $r=0$. Away from the origin, convergence of the derivatives follows from the continuous-dependence result already proved. Consequently, the convergence holds in $C^1(I)$. This completes the proof.
\end{proof}

\begin{lemma}\label{lem:log-global}
For every $\alpha>0$, the solution of \eqref{eq:intro-shooting-unified} exists globally on $[0,\infty)$ and is unique within the class of regular radial solutions. More generally, no nontrivial {regular} radial solution {in the class considered here} can have a double zero, either at an interior point or at an endpoint of its interval of definition, with derivatives at endpoints understood in the one-sided sense. Consequently, every zero of a nontrivial shooting solution is simple.
Moreover, if $u(z)=0$ and $u'(z)\ne0$, then
\begin{equation}\label{eq:logL1}
 \log u^2(r)+2=2\log|r-z|+O(1)
 \quad\hbox{as }r\to z,
\end{equation}
and therefore $\log u^2+2\in L^1_{\rm loc}$.
\end{lemma}

\begin{proof}
Local existence near the origin follows from the integral equation
\begin{equation}\label{eq:integralu}
 u(r)=\alpha-\int_0^r t^{1-n}\int_0^t s^{n-1}f(u(s))\,ds\,dt.
\end{equation}
Since $\alpha>0$, the nonlinearity $f$ is $C^1$ in a neighborhood of $\alpha$. The standard theory for regular radial initial-value problems therefore yields a unique local solution satisfying \eqref{eq:integralu}. At every point where $u\ne0$, the function $f$ is $C^1$, and hence the usual ODE existence and uniqueness theorem applies.

We next prove uniqueness across a simple zero. Let $z>0$, and suppose that $u$ and $\tilde u$ are two continuations across $z$ satisfying
\[
 u(z)=\tilde u(z)=0,
 \qquad
 u'(z)=\tilde u'(z)=d\ne0.
\]
Then there exist $c>0$ and small $\delta>0$ such that
\[
 |(1-\theta)u(r)+\theta\tilde u(r)|\geq c|r-z|
\]
for every $0<|r-z|<\delta$ and every $\theta\in[0,1]$. In particular, $u$ and $\tilde u$ have the same sign on each side of $z$.
Set $w=u-\tilde u$. Then
\[
 w''+\frac{n-1}{r}w'+A(r)w=0,
\]
where
\begin{equation*}
  A(r)=\int_0^1
f'\bigl((1-\theta) u(r)+\theta \tilde u(r)\bigr)\,d\theta.
\end{equation*}
Since $f'(s)=\log s^2+2$ for $s\ne 0$, the preceding lower bound
yields
\begin{equation*}
|A(r)|\le C\bigl(1+|\log|r-z||\bigr)
\end{equation*}
 for $0<|r-z|<\delta$.
Hence $A\in L^1(z-\delta,z+\delta)$. Moreover, $w(z)=w'(z)=0$. Lemma~\ref{lem:log-L1} therefore implies that $w\equiv0$ in a neighborhood of $z$. The usual uniqueness theorem away from the zero then extends this identity throughout the common interval of definition. Thus the continuation across a simple zero is unique.

If $u(z)=0$ and $u'(z)\ne0$, then $u(r)=(r-z)\eta(r)$, where $\eta$ is continuous near $z$ and satisfies $\eta(z)=u'(z)\ne0$. Hence $\eta$ is bounded away from zero in a neighborhood of $z$, and therefore
\[
 \log u^2(r)+2=2\log|r-z|+\log \eta^2(r)+2=2\log|r-z|+O(1)
 \quad\hbox{as }r\to z.
\]
This proves \eqref{eq:logL1}. Since $\log|r-z|$ is locally integrable, it also follows that $\log u^2+2\in L^1_{\rm loc}$.

It remains to rule out double zeros.  This is a local argument and
does not rely on the shooting condition $u(0)>0$.  Suppose that a nontrivial solution has a double zero. By replacing the double zero, if necessary, with an endpoint of a maximal interval on which $u\equiv0$, we may assume that
\[
u(z)=u'(z)=0,
\]
and that $u$ is nontrivial arbitrarily close to $z$ on at least one side.
Assume first that $u$ is nontrivial arbitrarily close to $z$ from the right. {Choose a connected component $(a,b_*)$ of
$\{r>z:u(r)\ne0\}$ whose left endpoint $a$ is arbitrarily close to
$z$. After replacing $b_*$ by a point $b\in(a,b_*)$, if necessary,
we may assume that $|u|<e^{1/2}$ on $(a,b)$.}
Then $u(a)=0$.  Since
$0\leq\frac12u'^2(a)=E(a)\le E(z)=0$, necessarily $u'(a)=0$ and $E(a)=0$.  After replacing $u$ by $-u$, if necessary, we may assume that
\[
 0<u(r)<e^{\frac{1}{2}}\quad \hbox{for }a<r<b.
\]
The monotonicity of the energy function $E$ gives $E(r)\le E(a)=0$ on $(a,b)$.
Consequently,
\[
 \frac12u'^2(r)\le -F(u(r))
 ={\frac12u^2(r)\bigl(1-\log u^2(r)\bigr)},
\]
and hence
 \begin{equation}\label{eqbdd}
 \frac{|u'(r)|}{u(r)\sqrt{1+|\log u^2(r)|}}\le C.
 \end{equation}
Fix $\sigma\in(0,e^{1/2})$ and define
\begin{equation*}
\Psi(\tau):=\int_\tau^\sigma\frac{ds}{s\sqrt{1+|\log s^2|}},
\qquad 0<\tau<\sigma.
\end{equation*}
The substitution $s=e^{-t}$ shows that $\Psi(\tau)\rightarrow+\infty$ as $\tau\downarrow0$, since the divergent part is comparable to $\int_0^\infty\frac{dt}{\sqrt{1+2t}}$. On the other hand, \eqref{eqbdd} gives $\left|\frac{d}{dr}\Psi(u(r))\right|\le C$ on $(a,b)$. Thus $\Psi(u(r))$ remains bounded as $r\downarrow a$, contradicting
$u(r)\to u(a)=0$. Hence $u$ cannot be nontrivial to the right of a double zero.

Having ruled out the possibility that $u$ is nontrivial arbitrarily close to $z$ from the right, we now consider the remaining case. Thus $u$ is nontrivial arbitrarily close to $z$ from the left. We first show that $u$ has no zeros arbitrarily close to $z$ from the left. Consequently, there exists $r_0<z$ such that $u$ has no zeros on
$[r_0,z)$. Suppose otherwise. Then there exist connected components $(a_j,b_j)\subset\{r<z:u(r)\ne0\}$ such that $b_j\uparrow z$. Then $u(a_j)=u(b_j)=0$.
On each component, $|u|$ attains a positive maximum at some point $c_j\in(a_j,b_j)$. Hence, $u'(c_j)=0$ and   $u(c_j)\ne0$.
Since $c_j\to z$ and $u(z)=0$, we also have $u(c_j)\to0$. In particular, $0<|u(c_j)|<e^{1/2}$ for all sufficiently large $j$. It follows that $E(c_j)=F(u(c_j))<0$. On the other hand,  $E(c_j)\ge E(z)=0$, which is a contradiction.  After decreasing $r_0$, if necessary, continuity and $u(z)=0$ allow us to assume that
$0<|u(r)|<1$ for $r_0\leq r<z$.
Thus $u$ has a fixed sign on $[r_0,z)$. Replacing $u$ by $-u$, if necessary, we may suppose that
\begin{equation*}
  0<u(r)<1
\quad\text{for }r_0\le r<z.
\end{equation*}
Since
\[
 (r^{n-1}u')'=-r^{n-1}f(u)>0,
 \qquad u'(z)=0,
\]
we have $u'<0$ on $[r_0,z)$.  Moreover,
\[
 u''=-\frac{n-1}{r}u'-f(u)>0.
\]
Thus $u'$ is increasing on $[r_0,z)$ and converges to $0$ as
$r\uparrow z$. In particular,
$|u'(s)|\le |u'(r)|$ for $r\le s<z$.  Since $E(z)=0$, integration of \eqref{eq:Eprime} gives
\[
 E(r)=\int_r^z\frac{n-1}{s}u'^2(s)\,ds
 \le \frac{n-1}{r_0}(z-r)u'^2(r),\qquad r_0\leq r<z.
\]
Choose $r_1\in(r_0,z)$ sufficiently close to $z$ so that
\[
\frac{n-1}{r_0}(z-r)\le\frac14
\quad\text{for }r_1\le r<z.
\]
Then
\[
 \frac14u'^2(r)\le -F(u(r))
 =\frac12u^2(r)\bigl(1-\log u^2(r)\bigr),\qquad r_1\le r<z.
\]
{Since $u>0$, $u'<0$, and $u(r)\to0$ as $r\uparrow z$, the
preceding differential inequality gives
$\left|\frac{d}{dr}\Psi(u(r))\right|\le C$ on $[r_1,z)$, with the
same function $\Psi$ as above. Thus $\Psi(u(r))$ remains bounded as
$r\uparrow z$, whereas $u(r)\to0$ implies
$\Psi(u(r))\to\infty$. This contradiction rules out a double zero
from the left.}
At the origin, only the right-sided argument is relevant. We conclude that every zero of a nontrivial solution, including a one-sided boundary zero on a finite interval, is simple.

Finally, finite-radius blow-up cannot occur. Since $F(u(r))\leq E(r)\leq E(0)$ and $F(s)\to\infty$ as $|s|\to\infty$, the function $u$ remains bounded on every {finite interval}. Moreover, if $M$ is a bound for $|u|$ on the interval under
consideration, then
\[
 \frac12u'(r)^2=E(r)-F(u(r))
 \le E(0)-\min_{|s|\le M}F(s).
\]
Thus $u'$ is bounded there as well. The standard ODE continuation criterion therefore excludes any finite maximal radius of existence. Hence the solution exists globally on $[0,\infty)$, and the proof is complete.
\end{proof}

\begin{lemma}\label{lem:log-C1}
For every $R>0$, the map $(0,\infty)\ni\alpha\mapsto u(\cdot,\alpha)\in C^1([0,R])$
is of class $C^1$. Its derivative with respect to $\alpha$,
\[
 v(\cdot,\alpha):=\partial_\alpha u(\cdot,\alpha)
\]
is the unique regular solution of
\begin{equation}\label{eq:var}
 v''+\frac{n-1}{r}v'+(\log u^2+2)v=0,
 \qquad
 v(0)=1,
 \qquad
 v'(0)=0,
\end{equation}
with the equation understood across the simple zeros of $u(\cdot,\alpha)$ in the sense of Lemma~\ref{lem:log-L1}.
If $z_j(\alpha_0)$ exists for some $\alpha_0>0$, then there is a neighborhood $J$ of
$\alpha_0$ such that $z_j(\alpha)$ is well defined in $J$ and
$z_j\in C^1(J)$. Moreover,
\begin{equation}\label{eq:zderivative}
 z_j'(\alpha)=-\frac{v(z_j(\alpha),\alpha)}{u_r(z_j(\alpha),\alpha)}\quad \hbox{for }\alpha\in J.
\end{equation}
\end{lemma}

\begin{proof}
Fix $R>0$ and an initial value $\alpha>0$.  Since the shooting profile
is global and its zeros are isolated, choose $\widehat R>R$ such that
$u(\widehat R,\alpha)\ne0$.  It suffices to prove all estimates on
$[0,\widehat R]$ and then restrict them to $[0,R]$; hence, after
renaming $\widehat R$, we may assume throughout the proof that
{$u(R,\alpha)\ne0$.} By Lemma~\ref{lem:log-global}, every zero of $u(\cdot,\alpha)$ is simple. Moreover, $u(\cdot,\alpha)$  has only finitely many zeros in $[0,R]$: otherwise, the zeros would have an accumulation point $z\in[0,R]$, and Rolle's theorem would yield critical points converging to $z$, so that $u(z)=u'(z)=0$, contradicting Lemma~\ref{lem:log-global}.

We also recall the stability of simple zeros. {If
$y_m\to y$ in $C^1([0,R])$, all zeros of $y$ are simple, and
$y(0)y(R)\ne0$, then, for all sufficiently large $m$, $y_m$ has
exactly one simple zero near each zero of $y$, no other zeros on
$[0,R]$, and these zeros converge to the corresponding zeros of $y$.}
Moreover,
\begin{equation}\label{eq:log-l1-convergence}
 \log y_m^2\longrightarrow \log y^2\quad\hbox{in }L^1([0,R]).
\end{equation}
Indeed, near a zero $z$ of $y$, we can write $y(r)=(r-z)\eta(r)$, where $\eta$ is bounded away from zero. By the implicit function theorem and the $C^1$-convergence, $y_m$ has a unique zero $z_m$ near $z$, with $z_m\to z$, and $y_m(r)=(r-z_m)\eta_m(r)$, where $|\eta_m|$ is uniformly bounded away from zero and infinity. Hence
\[
 \log y_m^2(r)=2\log|r-z_m|+\log\eta_m^2(r).
\]
Since the family $\log|r-z_m|$ is uniformly integrable and converges in $L^1$ to
$\log|r-z|$, while the convergence is uniform away from the zeros, we obtain \eqref{eq:log-l1-convergence}.
The same decomposition also controls the divided differences of $f(s)=s\log s^2$. For $a\ne b$ close to zero, we have
\[
 \left|\frac{f(a)-f(b)}{a-b}\right|
 \leq C\bigl(1+|\log |a||+|\log |b||\bigr).
\]
Indeed, if $a$ and $b$ have the same sign, the mean value theorem gives
\[
\frac{f(a)-f(b)}{a-b}=f'(\xi)
=\log \xi^2+2
\]
for some $\xi$ between $a$ and $b$, and the right-hand side is bounded by the
above expression. If $a$ and $b$ have opposite signs, then
$|a-b|=|a|+|b|$, and the estimate follows from
\[
|f(s)|\leq C|s|\bigl(1+|\log|s||\bigr)
\]
near the origin. Away from zero, the quotient is bounded by the usual mean value theorem.
Applying this estimate with $a=y_m(r)$ and $b=y(r)$, the factorization near a
simple zero yields the integrable majorant  $\bigl(1+|\log|r-z||+|\log|r-z_m||\bigr)$.
Therefore
\begin{equation}\label{eq:log-dd-l1-convergence}
 {\frac{f(y_m)-f(y)}{y_m-y}}
 \longrightarrow \log y^2+2\quad\hbox{in }L^1([0,R]),
\end{equation}
{where the quotient is assigned $f'(y)$ when $y_m=y\ne0$
and any finite value when $y_m=y=0$.}
{The quotients converge pointwise almost everywhere, and
the preceding logarithmic majorants form a uniformly integrable
family. Vitali's theorem therefore gives the asserted $L^1$
convergence; away from the zeros the convergence is uniform.}

We next prove continuous dependence on the shooting height in $C^1$. Let $\alpha_m\to\alpha$. On $[0,R]$,  the initial energies $E(0,\alpha_m)=F(\alpha_m)$ remain bounded. Since $F(s)\to\infty$ as $|s|\to\infty$, the energy identity yields a uniform bound for $\{u(\cdot,\alpha_m)\}$. Since $F$ is bounded below, using \eqref{eq:energy} once more gives a uniform bound for $\{u_r(\cdot,\alpha_m)\}$. Consequently, $f(u(\cdot,\alpha_m))$ is uniformly bounded on $[0,R]$. Differentiating \eqref{eq:integralu}, we obtain, for $r>0$,
\begin{equation}\label{eq:diff}
  u_r(r,\alpha_m)=-r^{1-n}\int_0^r s^{n-1}f(u(s,\alpha_m))\,ds,
 \qquad u_r(0,\alpha_m)=0.
\end{equation}
In particular, $|u_r(r,\alpha_m)|\le C_0r$ near the origin.
We next show that $\{u(\cdot,\alpha_m)\}$ and $\{u_r(\cdot,\alpha_m)\}$ are  equicontinuous. Near the origin, this follows from $|u_r(r,\alpha_m)|\le Cr$. On the other hand, if $0<\delta\leq r_1<r_2\leq R$, then \eqref{eq:diff} yields
\[
\begin{aligned}
 |u_r(r_1,\alpha_m)-u_r(r_2,\alpha_m)|
 &\leq \left|r_2^{1-n}-r_1^{1-n}\right|
       \int_0^{r_1}\tau^{n-1}|f(u(\tau,\alpha_m))|\,d\tau  \\
 &\quad +r_2^{1-n}\int_{r_1}^{r_2}\tau^{n-1}|f(u(\tau,\alpha_m))|\,d\tau
 \leq C_\delta |r_2-r_1|,
\end{aligned}
\]
where $C_\delta$ is independent of $\alpha_m$. {Given $\varepsilon>0$, choose $\delta>0$ so small that
$2C_0\delta<\varepsilon/2$, where $|u_r(r,\alpha_m)|\le C_0r$ on
$[0,\delta]$.  On $[\delta,R]$, use the uniform Lipschitz estimate
above to choose $\eta>0$ such that
$C_\delta\eta<\varepsilon/2$.  If $r_1<\delta<r_2$, split the
difference at $\delta$.  This gives a uniform modulus of continuity
for $\{u_r(\cdot,\alpha_m)\}$ on $[0,R]$.} The uniform boundedness of $\{u_r(\cdot,\alpha_m)\}$ further implies that $\{u(\cdot,\alpha_m)\}$ is equicontinuous.
It follows from the Arzel\`{a}--Ascoli theorem that $\{u(\cdot,\alpha_m)\}$ is
relatively compact in $C^1([0,R])$. Let $\tilde u$ be the $C^1$-limit of
a subsequence. Passing to the limit in the integral equation \eqref{eq:integralu} shows that $\tilde u$ solves the regular initial value problem \eqref{eq:intro-shooting-unified}. By Lemma~\ref{lem:log-global}, $\tilde u=u(\cdot,\alpha)$. Since every convergent subsequence has the same limit, the whole family satisfies  $u(\cdot,\alpha_m)\rightarrow u(\cdot,\alpha)$ in $C^1([0,R])$.

For $h\ne0$, set
\[
 w_h(r):=\frac{u(r,\alpha+h)-u(r,\alpha)}{h}.
\]
Then
\[
 w_h''+\frac{n-1}{r}w_h'+a_h(r)w_h=0,\qquad w_h(0)=1,\qquad w_h'(0)=0,
\]
where
\[
a_h(r)=
\begin{cases}
\displaystyle
\frac{f(u(r,\alpha+h))-f(u(r,\alpha))}
     {u(r,\alpha+h)-u(r,\alpha)},
& u(r,\alpha+h)\neq u(r,\alpha),\\[1ex]
f'(u(r,\alpha)),
& u(r,\alpha+h)=u(r,\alpha)\neq0,\\
0,
& u(r,\alpha+h)=u(r,\alpha)=0.
\end{cases}
\]
At points where $u(r,\alpha+h)=u(r,\alpha)=0$, we set $a_h(r)=0$. This choice is immaterial, since these points belong to the zero set of $u(\cdot,\alpha)$, which has measure zero. With arbitrary finite values assigned to the limiting coefficient at the zeros of $u(\cdot,\alpha)$, the logarithmic $L^1$-stability result \eqref{eq:log-dd-l1-convergence} gives
\[
 a_h\longrightarrow \log u^2(\cdot,\alpha)+2\quad\hbox{in }L^1([0,R]).
\]
Lemma~\ref{lem:log-L1} therefore implies that $w_h$ converges in $C^1([0,R])$ to the unique regular solution $v(\cdot,\alpha)$ of \eqref{eq:var}. Hence $\partial_\alpha u(\cdot,\alpha)=v(\cdot,\alpha)$ in $C^1([0,R])$. We next verify the continuity of this derivative with respect to $\alpha$. If $\alpha_m\to\alpha$, then \eqref{eq:log-l1-convergence} gives
\[
\log u^2(\cdot,\alpha_m)+2
\longrightarrow
\log u^2(\cdot,\alpha)+2
\quad\text{in }L^1([0,R]).
\]
Applying Lemma~\ref{lem:log-L1} once more, we obtain
$v(\cdot,\alpha_m)\rightarrow v(\cdot,\alpha)$
in $C^1([0,R])$.
Thus the map $\alpha\mapsto u(\cdot,\alpha)$ is continuously differentiable as a map into $C^1([0,R])$.

Finally, suppose that $z_j(\alpha_0)$ exists for some $\alpha_0>0$. By Lemma~\ref{lem:log-global},
$u_r(z_j(\alpha_0),\alpha_0)\neq0$.
Hence the implicit function theorem, together with the {continuous dependence of $u$ on $\alpha$ in $C^1$}, yields a neighborhood $J$ of $\alpha_0$ on which $z_j(\alpha)$ is well defined and belongs to $C^1(J)$. Differentiating $u(z_j(\alpha),\alpha)=0$ with respect to $\alpha$ gives \eqref{eq:zderivative}.
This completes the proof.
\end{proof}

\subsection{Whole-space phase transition}
We now state the whole-space shooting result used below. This is the only ingredient in the proof that relies on the recent phase-transition theory for the variation equation. We formulate it in the form needed for the finite-ball zero-curve analysis.

\begin{theorem}\label{thm:log-whole}
Let $n\geq2$ and $u(\cdot,\alpha)$ be the solution of \eqref{eq:intro-shooting-unified}.  There exists a strictly increasing sequence
\[
 e^{\frac{n}{2}}=\alpha_0<\alpha_1<\alpha_2<\cdots,
 \qquad \alpha_k\to\infty,
\]
which is precisely the sequence $\{\alpha_k^{\log}\}$ introduced in
\eqref{eq:alphak}, with the following properties.
\begin{enumerate}[label=\rm(\roman*)]
\item $u(\cdot,\alpha_0)$ is the Gausson given by \eqref{eq:intro-gausson}.
\item For every $k\geq1$, $u(\cdot,\alpha_k)$ is the unique whole-space bound state
with exactly $k$ simple zeros.
\item If $\alpha\in(\alpha_k,\alpha_{k+1})$, then $u(\cdot,\alpha)$ has exactly $k+1$ simple zeros and,
after its last zero, oscillates about either $1$ or $-1$.
\item If $\alpha<\alpha_0$ and $\alpha\ne1$, then $u(r,\alpha)>0$ for every $r>0$ and $u(\cdot,\alpha)$ oscillates about $1$.
\item If $u(\cdot,\alpha)$ has at least $j$ zeros, then the corresponding variation $v(\cdot,\alpha)=\partial_\alpha u(\cdot,\alpha)$ satisfies
\begin{equation}\label{eq:uvsign}
 u_r(z_j(\alpha),\alpha)v(z_j(\alpha),\alpha)>0.
\end{equation}
Consequently, $z_j'(\alpha)<0$ wherever the $j$-th zero curve is defined.
\end{enumerate}
\end{theorem}

\begin{proof}
Assertions (i)--(iv) follow from Theorem 1.2 of Liu--Sun--Zou
\cite{LiuSunZou}, which is valid for every $n\geq2$ {and provides the
complete shooting classification needed here}.  For $n\geq3$,  \cite[Proposition 3.5]{LiuSunZou} shows that $v$ has exactly one zero in $(0,z_1(\alpha))$ and exactly one zero between each pair of consecutive zeros of $u$.
For $n=2$, the same interlacing property follows from Proposition~\ref{prop:log-planar}, where the only dimension-dependent positivity step in the phase-transition argument is replaced by its planar counterpart.

It remains to deduce the sign relation. Since $u(0,\alpha)>0$ and all zeros of $u$ are simple,
\[
 \operatorname{sgn}u_r(z_j(\alpha),\alpha)=(-1)^j.
\]
Moreover, $v(0,\alpha)=1$, and the above interlacing shows that $v$ changes sign exactly once before $z_1(\alpha)$ and once between every two consecutive zeros of $u$. Hence
\[
 \operatorname{sgn}v(z_j(\alpha),\alpha)=(-1)^j.
\]
Therefore, \eqref{eq:uvsign} holds, and \eqref{eq:zderivative} yields $z_j'(\alpha)<0$.
\end{proof}

\begin{proposition}\label{prop:log-planar}
Let $n=2$, and let $u(\cdot,\alpha)$ be a nodal solution of \eqref{eq:intro-shooting-unified}. Set $v(\cdot,\alpha)=\partial_\alpha u(\cdot,\alpha)$. Then $v(\cdot,\alpha)$ has exactly one zero in $(0,z_1(\alpha))$ and, for every $j\geq2$ such that $z_j(\alpha)$ exists, exactly one zero in $(z_{j-1}(\alpha),z_j(\alpha))$.
Consequently,
\begin{equation}\label{eq:planar-strict-zero-sign}
 u_r(z_j(\alpha),\alpha)v(z_j(\alpha),\alpha)>0
\end{equation}
for every zero $z_j(\alpha)$ of $u(\cdot,\alpha)$.
\end{proposition}

\begin{proof}
We provide only the planar replacement for the dimension-dependent step in the phase-transition argument of \cite[Section~3]{LiuSunZou}. All the remaining sign arguments, as well as the analysis of the first phase, require only $n\geq2$ and the differential identities stated below.

Write
\[
 f(u)=u\log u^2,
 \qquad
 F(u)=\frac12u^2(\log u^2-1),
 \qquad
 E=\frac12u'^2+F(u).
\]
For $n=2$, introduce
\begin{align*}
 Q&=r^2\bigl(u'v'+f(u)v\bigr),
 &Q_j&=Q+jru'v\quad(j=1,2),\\
 M&=r(u'v-uv'),
 &T_2&=Q-(\log u^2-1)M.
\end{align*}
On every interval containing no zero of $u$, direct differentiation gives
\begin{align*}
 Q'&=2rf(u)v,
 &Q_j'&=r\bigl(ju'v'+(2-j)f(u)v\bigr),\\
 M'&=2ruv,
 &T_2'&=-\frac{2u'}uM+2ruv.
\end{align*}
The identities for $Q$, $Q_j$, and $M$ extend across a simple zero in
the integral sense.  The quantity $T_2$ is used only on components of
$\{u\ne0\}$; indeed, $T_2$ itself has a logarithmic divergence at a
zero. The bridge quantity used below has a continuous extension through that
zero, and its derivative is integrable there; both assertions are verified
explicitly before the bridge identity is integrated.

For completeness, we identify all points at which the dimension enters the
phase induction.  The first-phase argument uses the energy identity
$E'=-(n-1)u'^2/r$, the signs of $f$, $F$, and $uf-2F$, and the regular
conditions $u'(0)=v'(0)=0$; each remains valid for $n=2$.  On a complete
phase, the production of the first zero of $v$, the exclusion of a second
zero before $\bar b_i$, and the propagation of the signs of $Q$, $Q_1$, and
$M$ use only the four differential identities displayed above and integration
over intervals whose endpoints are simple zeros or critical points.  No
factor $n-2$ is divided out in those steps.  The sole remaining
dimension-dependent quantity is the bridge term $Q_n$ in
\eqref{eq:planar-Baprime}.  In dimension two it is exactly $Q_2$, whose
positivity is obtained before the bridge identity is used.  Thus the
induction below verifies the complete planar argument rather than assuming a
higher-dimensional positivity statement.

Set $c_0=0$. In a complete nodal phase $[c_{i-1},c_i]$, the endpoints
are consecutive critical points of $u$, and $z_i$ is the unique zero of
$u$ in the phase.  Introduce the uniquely determined points
\[
 c_{i-1}<b_i<r_i<z_i<\bar r_i<\bar b_i<c_i,
 \qquad
 |u(b_i)|=|u(\bar b_i)|=e^{\frac{1}{2}},
 \qquad
 |u(r_i)|=|u(\bar r_i)|=1.
\]

\medskip
\noindent\emph{Base phase.}
The first-phase calculation starts from $v(0)=1$ and
$u'(0)=v'(0)=Q(0)=M(0)=Q_1(0)=Q_2(0)=0$; thus it cannot be obtained by
substituting $c_0=0$ into a strict induction hypothesis. We record the full
base conclusion. There is a first zero $\tau_1\in(0,r_1)$ of $v$, and
\begin{equation}\label{eq:planar-base-v}
 v>0\quad\hbox{on }(0,\tau_1),
 \qquad v<0\quad\hbox{on }(\tau_1,c_1].
\end{equation}
Moreover,
\begin{align}
 Q&>0\quad \hbox{on }(0,z_1]\cup[\bar b_1,c_1],\label{eq:planar-base-Q}\\
 M,Q_1,Q_2&>0\quad\hbox{on }(0,c_1],\label{eq:planar-base-M}\\
 T_2&>0\quad\hbox{on }(0,b_1].\label{eq:planar-base-T}
\end{align}
We give the missing dimension check rather than importing a strict induction
hypothesis at the origin. The dimension-free first-phase calculation
\cite[Proposition~3.2]{LiuSunZou} gives a unique zero
$\tau_1\in(0,r_1)$ before $z_1$, together with
$Q,M>0$ on $(0,z_1]$. On $(0,\tau_1)$ one has
$u>1$, $f(u)>0$, $f'(u)>0$, $v>0$, and $u',v'<0$. Hence
\[
 Q_1'=r(u'v'+f(u)v)>0,
 \qquad Q_2'=2ru'v'>0,
\]
so $Q_1,Q_2>0$ there.

For $a\in\mathbb R$, set
\begin{equation*}
 F_a(u):=F(u)-\frac a2\bigl(uf(u)-2F(u)\bigr)
       =\frac12u^2(\log u^2-1-a)
\end{equation*}
and
\begin{equation*}
 B_a:=Q-aM-2F_a(u)\frac{rv}{u'}.
\end{equation*}
The bridge identity of \cite{LiuSunZou} becomes
\begin{equation}\label{eq:planar-Baprime}
 B_a'=-2F_a(u)\frac{Q_n}{r u'^2}
     =-2F_a(u)\frac{Q_2}{r u'^2},
\end{equation}
because
\begin{equation*}
 Q_n=Q_2,\qquad n=2.
\end{equation*}

We next continue these signs beyond $z_1$. If $Q_1$ had a first zero
$\tilde z\in(z_1,\bar b_1]$ while $v<0$, then
$u'v>0$ and $Q_2=Q_1+ru'v>Q_1>0$ before $\tilde z$. Since
$F(u)<0$ on $(z_1,\bar b_1)$, the identity
\[
 B_0'=-2F(u)\frac{Q_2}{ru'^2}>0
\]
would give $B_0(\tilde z)>B_0(z_1)=Q(z_1)>0$. On the other hand,
$Q_1(\tilde z)=0$ and the energy identity give
\[
 B_0(\tilde z)
 =-\frac{2\tilde zv(\tilde z)}{u'(\tilde z)}
 E(\tilde z)<0,
\]
a contradiction. Thus $Q_1>0$ on $[z_1,\bar b_1]$, which also excludes
a second zero of $v$ there and gives
$Q(\bar b_1)>Q(z_1)>0$. If $v$ had a first zero in
$(\bar b_1,c_1]$, then $u< -e^{1/2}$ and $v<0$ before that zero, so
$Q'=2rf(u)v>0$. At the first zero, however,
$Q=r^2u'v'<0$, again a contradiction. This proves
\eqref{eq:planar-base-v} and \eqref{eq:planar-base-Q}.

On $[z_1,c_1]$, $uv>0$, and therefore $M'=2ruv>0$; together with
$M(z_1)>0$ this gives $M>0$. On
$[\tau_1,z_1]\cup[\bar b_1,c_1)$ one has $u'v>0$ and hence
$Q_2>Q_1>Q>0$, while on $[z_1,\bar b_1]$ the preceding first-contact
argument gives $Q_2>Q_1>0$. At $c_1$, $u'(c_1)=0$, so
$Q_1(c_1)=Q_2(c_1)=Q(c_1)>0$. This proves
\eqref{eq:planar-base-M}. Finally, define on $(0,z_1)$
\[
 T_1:=Q-(\log u^2)M.
\]
The first-phase identity gives $T_1,T_1'>0$ there, while
$T_2=T_1+M>T_1$; this gives \eqref{eq:planar-base-T}. Every comparison
above uses $Q_2$ itself and no positive multiple of $n-2$. The bridge
argument below,
applied with $i=1$, upgrades \eqref{eq:planar-base-T} to
$T_2>0$ on $[\bar b_1,c_1]$. Hence
\begin{equation}\label{eq:planar-base-end}
 Q(c_1)>0,
 \qquad M(c_1)>0,
 \qquad T_2(c_1)>0,
\end{equation}
which is the genuine base datum for the phase induction.

\medskip
\noindent\emph{Inductive phases.}
Let $i\ge2$ and assume
\begin{equation}\label{ih}
Q(c_{i-1})>0,\qquad M(c_{i-1})>0,\qquad T_2(c_{i-1})>0.
\end{equation}
{The dimension-independent production step of the phase argument,
applied after $c_{i-1}$, yields a first zero $\tau_i$ of $v$ in
$(c_{i-1},r_i)$.}  We now apply Steps~1--3 in the
proof of \cite[Proposition~3.4]{LiuSunZou}.  Inspection of those steps shows
that they use only the displayed differential identities, $n>1$, and the
induction data \eqref{ih}; the quantity $Q_n$ is not used there.  They give
\begin{align*}
T_2&>0
\quad\hbox{on }[c_{i-1},\max\{\tau_i,b_i\}],\\
Q,Q_1,Q_2&>0
\quad\hbox{on }[c_{i-1},\max\{\tau_i,b_i\}],
\end{align*}
and extend \(Q_1>0\) from \(\tau_i\) up to \(\bar b_i\).

Integrating $Q'=2rf(u)v$ on the incoming and outgoing pieces and using
the matching values of $|u|$ gives $Q(\bar b_i)>Q(b_i)>0$.
The sign of $Q_1$ then excludes any further zero of $v$ on
$(\tau_i,c_i]$.
Consequently,
\[
u'v>0\quad\hbox{on }(\tau_i,c_i],
\qquad
Q>0\quad\hbox{on }[\bar b_i,c_i],
\]
and hence \(Q_1>0\) throughout \([c_{i-1},c_i]\). Moreover,
\[
Q_2=Q_1+ru'v>Q_1>0
\quad\hbox{on }(\tau_i,c_i].
\]
The identity $M'(r)=2ruv$ also
shows that \(M\) increases from \(M(c_{i-1})>0\) to \(M(\tau_i)\),
then decreases to $M(z_i)=z_iu'v>0$
and finally increases up to \(c_i\). Hence
\[
M>0\quad\hbox{on }[c_{i-1},c_i].
\]
Thus, under the induction hypothesis \eqref{ih}, the preceding arguments have
already yielded
\begin{equation}\label{eq:planar-induction-positive}
Q_1,Q_2,M>0
\quad\hbox{on }[c_{i-1},c_i],
\qquad
T_2>0
\quad\hbox{on }[c_{i-1},b_i],
\end{equation}
as well as
\begin{equation*}
Q>0
\quad\hbox{on }[c_{i-1},b_i]\cup[\bar b_i,c_i],
\qquad
Q(\bar b_i)>Q(b_i)>0.
\end{equation*}
For $i=1$, the same statements follow from
\eqref{eq:planar-base-Q}--\eqref{eq:planar-base-T} and the first-phase
comparison $Q(\bar b_1)>Q(b_1)$. Thus the remaining bridge argument applies
both to the base phase and to every inductive phase.
It remains only to provide the planar proof that
\begin{equation}\label{key}
T_2>0\quad\hbox{on }[\bar b_i,c_i].
\end{equation}

We now establish the missing positivity on the outgoing part of the phase.
Since $F(u(\bar b_i))=0$, we have $T_2(\bar b_i)=Q(\bar b_i)>0$. Moreover, $T_2'(c_i)=2c_i uv>0$. Let $\bar t_i\in(\bar b_i,c_i)$ be a critical point of $T_2$. The algebraic identity
\cite[Lemma~3.4]{LiuSunZou}
\begin{equation*}
 T_2=B_0+\frac{F(u)}{uu'}T_2'
\end{equation*}
together with \eqref{eq:planar-induction-positive} yields
\begin{equation*}
 T_2''(\bar t_i)
 =\frac{2u(\bar t_i)}{\bar t_i u'(\bar t_i)}Q_2(\bar t_i)>0,
\end{equation*}
because $uu'>0$ on the outgoing part of the phase. Thus every critical point of $T_2$ in $(\bar b_i,c_i)$ is a strict local minimum. Consequently, there can be at most one such critical point, since two distinct local minima would force the existence of a local maximum between them.
More explicitly, $T_2''>0$ makes $T_2'$ change from negative to positive
at each critical point.  Two such changes would force an intermediate zero
of $T_2'$ at which the change is from positive to nonpositive, contradicting
the same inequality $T_2''>0$.
If $T_2$ has no critical point in $(\bar b_i,c_i)$, then $T_2'>0$ throughout $(\bar b_i,c_i)$, and hence
$T_2(r)>T_2(\bar b_i)>0$ for $r\in(\bar b_i,c_i]$.

Suppose now that the unique critical point, necessarily a strict minimum,
occurs at $\bar t_i$. Since the energy decreases strictly on a nonconstant
phase and $F$ is strictly increasing as a function of $|u|$ above
$e^{1/2}$, one has
\[
 |u(c_{i-1})|>|u(c_i)|>|u(\bar t_i)|>e^{\frac{1}{2}}.
\]
The strict monotonicity of $u$ on the incoming part of the phase therefore
gives a unique matching point $t_i\in(c_{i-1},b_i)$ such that
\[
 u(t_i)=-u(\bar t_i),
 \qquad
 U_i:=|u(t_i)|=|u(\bar t_i)|>e^{\frac{1}{2}},
\]
and set $a_i:=\log U_i^2-1$. Then
\begin{equation}\label{eq:planar-Fai-sign}
 -2F_{a_i}(s)=s^2\log\frac{U_i^2}{s^2}\geq0,
 \qquad |s|\leq U_i,
\end{equation}
where the expression is understood by continuity to be zero at $s=0$. Moreover,
\[
 B_{a_i}(t_i)=T_2(t_i),
 \qquad
 B_{a_i}(\bar t_i)=T_2(\bar t_i).
\]

We now justify integration through the simple zero $z_i$. Put
$d_i:=u'(z_i)\ne0$. The nonlinear equation, the integral linearized
equation, and $\log|r-z_i|\in L^1_{\rm loc}$ give
\begin{align*}
 u(r)&=d_i(r-z_i)+O(|r-z_i|^2),
 &u'(r)&=d_i+O(|r-z_i|),\\
 v(r)&=v(z_i)+O(|r-z_i|),
 &Q_2(r)&=Q_2(z_i)+O(|r-z_i|\,|\log|r-z_i||).
\end{align*}
Consequently,
\[
 F_{a_i}(u(r))
 =O\bigl(|r-z_i|^2(1+|\log|r-z_i||)\bigr),
\]
and $rv/u'$ remains bounded. Hence the last term in $B_{a_i}$ tends to
zero at $z_i$, while $Q-a_iM$ is continuous there. Thus $B_{a_i}$ has the
same finite one-sided limit and extends continuously across $z_i$. Moreover,
\[
 \left|-2F_{a_i}(u(r))\frac{Q_2(r)}{r u'(r)^2}\right|
 \le C|r-z_i|^2(1+|\log|r-z_i||),
\]
which is integrable. Integrate \eqref{eq:planar-Baprime} first on
$[t_i,z_i-\varepsilon]$ and $[z_i+\varepsilon,\bar t_i]$ and then let
$\varepsilon\downarrow0$. Using
\eqref{eq:planar-induction-positive} for $i\ge2$, or
\eqref{eq:planar-base-M} for $i=1$, together with
\eqref{eq:planar-Fai-sign}, gives
\begin{align*}
 T_2(\bar t_i)-T_2(t_i)
 &=\int_{t_i}^{\bar t_i}
   \bigl[-2F_{a_i}(u(r))\bigr]
   \frac{Q_2(r)}{r u'^2(r)}\,dr>0.
\end{align*}
The inequality is strict because $Q_2>0$ and
$-F_{a_i}(u)>0$ away from the matching endpoints and the nodal zero. Since
$t_i<b_i$, either \eqref{eq:planar-base-T} or the incoming induction
hypothesis gives $T_2(t_i)>0$. Therefore
\[
 T_2(\bar t_i)>T_2(t_i)>0.
\]
Thus the unique minimum of $T_2$ on the outgoing part is positive, and consequently
\begin{equation*}
 T_2(r)>0\quad\hbox{for every }r\in[\bar b_i,c_i].
\end{equation*}
This proves \eqref{key}.

For $i=1$, the positivity just proved gives the base datum
\eqref{eq:planar-base-end}. For every $i\ge2$, it propagates the induction
hypotheses from $c_{i-1}$ to $c_i$. Repetition over the complete phases gives exactly
one zero of $v$ in each such phase. {On the final incomplete phase, the
same production step gives the first zero of $v$, while $Q_1>0$ excludes a
second zero by the same first-contact argument used above.} Consequently, $v$ has exactly one zero in
each interval $(z_{j-1},z_j)$, $j\geq1$, where $z_0:=0$. It follows that
 $\operatorname{sgn}u_r(z_j)=(-1)^j$ and $\operatorname{sgn}v(z_j)=(-1)^j$.
Therefore, $u_r(z_j)v(z_j)>0$,
which proves \eqref{eq:planar-strict-zero-sign}.
\end{proof}

\subsection{Endpoint behavior of zero curves}
We next determine the exact domain of each zero curve and its limiting behavior as the shooting height approaches either endpoint of that domain. The argument below does not use the whole-space interlacing result and applies in every dimension $n\geq2$, provided that the corresponding zero $z_j(\alpha)$ exists.

\begin{lemma}\label{lem:log-left}
Fix $j\geq1$. The zero curve $z_j(\alpha)$ is defined precisely for $\alpha\in(\alpha_{j-1},\infty)$. Moreover, $z_j(\alpha)$ is $C^1$, strictly decreasing, and satisfies
\begin{equation}\label{eq:z-infty}
 \lim_{\alpha\downarrow\alpha_{j-1}}z_j(\alpha)=\infty.
\end{equation}
\end{lemma}

\begin{proof}
By Theorem~\ref{thm:log-whole}, the solution $u(\cdot,\alpha)$ has at least $j$ zeros if and only if $\alpha>\alpha_{j-1}$. At $\alpha=\alpha_{j-1}$, it is the whole-space $(j-1)$-node bound state;
 when $j=1$, we adopt the convention that $\alpha_0$ corresponds to the Gausson. Thus the $j$-th zero does not exist at $\alpha=\alpha_{j-1}$, and $z_j$ is defined precisely on $(\alpha_{j-1},\infty)$.
 Lemma~\ref{lem:log-C1} and Theorem~\ref{thm:log-whole}(v) then imply that $z_j$ is $C^1$ and strictly decreasing on this interval.

It remains to prove \eqref{eq:z-infty}. Suppose, to the contrary, that the conclusion fails. Then there exist $L_0>0$ and a sequence $\alpha_m\downarrow\alpha_{j-1}$ such that $z_j(\alpha_m)\leq L_0$.
When $j\geq2$, let $0<\zeta_1<\cdots<\zeta_{j-1}$ denote the zeros of the whole-space $(j-1)$-node bound state $u(\cdot,\alpha_{j-1})$. When $j=1$, the function $u(\cdot,\alpha_{0})$ is the Gausson and has no zeros.
 Choose
 \[
L>
\begin{cases}
L_0,\qquad&j=1,\\
\max\{L_0,\zeta_{j-1}+1\},\qquad&j\geq2.
\end{cases}
\]
Since each $\zeta_\ell$ is a simple zero of $u(\cdot,\alpha_{j-1})$,
we may choose pairwise disjoint intervals
$I_\ell=(\zeta_\ell-\delta_\ell,\zeta_\ell+\delta_\ell)\subset(0,L)$ so small that
{$u_r(\cdot,\alpha_{j-1})$ has a fixed nonzero sign
throughout each $I_\ell$.} Hence $u(\cdot,\alpha_{j-1})$ is strictly
monotone on $I_\ell$, and its values at the two endpoints of
$I_\ell$ have opposite signs. Outside these intervals, $u(\cdot,\alpha_{j-1})$ has no zeros and is therefore uniformly bounded away from zero; for $j=1$, this simply follows from the strict positivity of the Gausson on $[0,L]$. By the {continuous dependence in $C^1([0,L])$} established in Lemma~\ref{lem:log-C1}, $u(\cdot,\alpha_m)\rightarrow u(\cdot,\alpha_{j-1})$ in $C^1([0,L])$.
Consequently, for all sufficiently large $m$, the function $u(\cdot,\alpha_m)$ has no zeros on {$[0,L]\setminus \bigcup_{\ell=1}^{j-1}I_{\ell}$},
while its derivative has a fixed nonzero sign on each $I_\ell$ and its values at the two endpoints of $I_\ell$ have opposite signs. It therefore has exactly one zero in each $I_\ell$. Thus $u(\cdot,\alpha_m)$  has exactly $j-1$ zeros in $[0,L]$ when $j\geq2$, and no zeros there when $j=1$.
This contradicts $z_j(\alpha_m)\leq L_0<L$, which asserts that $u(\cdot,\alpha_m)$ has at least $j$ zeros in $[0,L]$. Hence \eqref{eq:z-infty} holds.
\end{proof}

We now turn to the behavior of the zero curves at the right endpoint of their domains.
Set
\[
 A:=\log\alpha^2,
 \qquad
 w_A(s):=\frac{u(s/\sqrt A,\alpha)}{\alpha}.
\]
We first obtain a uniform bound from the monotonicity of the energy. Indeed, if $\alpha>1$, then $F(u(r,\alpha))\leq E(r)<E(0)=F(\alpha)$ for $r>0$.
Since $F$ is even and strictly increasing on $(1,\infty)$, we have
\begin{equation*}
 |u(r,\alpha)|<\alpha
 \quad\text{for all }r\geq0.
\end{equation*}
Consequently, $|w_A|\leq1$ for $s\geq0$. A direct rescaling of the shooting equation gives
\begin{equation}
 \begin{cases}
 {w_A''+\dfrac{n-1}{s}w_A'+w_A+\dfrac{1}{A} w_A\log w_A^2=0,\qquad  s>0,}\\[4pt]
 w_A(0)=1,\qquad w_A'(0)=0,
 \end{cases}
 \label{eq:scaled}
\end{equation}
where $w_A\log w_A^2$ is understood to be zero at the zeros of $w_A$.

\begin{lemma}\label{lem:log-bessel}
As $\alpha\to\infty$, equivalently $A\to\infty$, one has
$w_A\to\Phi_n$ in $C^1([0,S])$ for every $S>0$. Let $\Psi_n$ be the
unique regular integral solution of
\begin{equation}\label{eq:log-second-profile}
 \Psi''+\frac{n-1}{s}\Psi'+\Psi
 =-\Phi_n\log\Phi_n^2,
 \qquad \Psi(0)=\Psi'(0)=0.
\end{equation}
Then
\begin{equation}\label{eq:log-second-profile-convergence}
 A(w_A-\Phi_n)\longrightarrow\Psi_n
 \quad\hbox{in }C^1([0,S])
\end{equation}
for every $S>0$. Consequently, if $\rho_j$ is the $j$-th positive zero
of $\Phi_n$, then
\begin{equation}\label{eq:zeroasymptotic}
 \sqrt A\,z_j(\alpha)
 =\rho_j+\frac{c_{j,n}}{A}+o(A^{-1}),
 \qquad
 c_{j,n}:=-\frac{\displaystyle\int_0^{\rho_j}
 s^{n-1}\Phi_n^2(s)\log\Phi_n^2(s)\,ds}
 {\rho_j^{n-1}\Phi_n'^2(\rho_j)}>0.
\end{equation}
In particular,
\begin{equation}\label{eq:z-zero}
 \lim_{\alpha\to\infty}z_j(\alpha)=0.
\end{equation}
\end{lemma}

\begin{proof}
Equation \eqref{eq:scaled} is equivalent to the integral equation
\begin{equation}\label{eq:integralw}
 w_A(s)=1-\int_0^s t^{1-n}\int_0^t \tau^{n-1}\left(w_A(\tau)+\frac1A w_A(\tau)\log w_A^2(\tau)\right)d\tau\,dt.
\end{equation}
Set
\[
 G_A(s):=w_A(s)+\frac1A w_A(s)\log w_A^2(s).
\]
The function $\xi\mapsto \xi\log\xi^2$ extends continuously to $\xi=0$ by setting its value there equal to zero. Since $|w_A|\leq 1$ and $A\rightarrow\infty$, it follows that
\begin{equation*}
|G_A|\leq |w_A|+\frac1A \big|w_A\log w_A^2\big|\leq 1+\frac{C}{A}\leq2.
\end{equation*}
 Differentiating \eqref{eq:integralw}, we obtain
\begin{equation*}
 w_A'(s)=-s^{1-n}\int_0^s \tau^{n-1}G_A(\tau)\,d\tau.
\end{equation*}
Consequently,
 \begin{equation*}
 |w_A'(s)|\leq 2s^{1-n}\int_0^s \tau^{n-1}\,d\tau \leq \frac{2}{n}s.
\end{equation*}
 In particular, the families $\{w_A\}$ and $\{w_A'\}$ are uniformly bounded on $[0,S]$. Near the origin, it follows from
$|w_A'(s)|\leq Cs$ that $\{w_A\}$ and $\{w_A'\}$ are  equicontinuous.
On every interval
$[\delta,S]$, $\delta>0$, the same formula gives a uniform modulus of continuity for $\{w_A'\}$. Together with the uniform bound on $\{w_A'\}$, this also gives equicontinuity of $\{w_A\}$.
Thus, after passing to a subsequence, we obtain a limit $w$ in $C^1([0,S])$.  Since
\begin{equation*}
\sup_{s\in[0,S]}\left|
\frac1A w_A(s)\log w_A^2(s)\right|\leq \frac{C}{A}\longrightarrow0\quad \text{as}~A\rightarrow\infty,
\end{equation*}
we may pass to the limit in \eqref{eq:integralw} and obtain
\[
 w(s)=1-\int_0^s t^{1-n}\int_0^t \tau^{n-1}w(\tau)\,d\tau\,dt,\qquad w(0)=1,\qquad
w'(0)=0.
\]
This is precisely the integral formulation of \eqref{eq:main-bessel}. The regular solution of \eqref{eq:main-bessel} is unique, hence $w=\Phi_n$. Since every convergent subsequence has the same limit, the whole family satisfies $w_A\to\Phi_n$ in $C^1([0,S])$ for every fixed $S>0$.

Set $Y_A=A(w_A-\Phi_n)$. Subtracting the equation for $\Phi_n$ from
\eqref{eq:scaled} gives
\begin{equation}\label{eq:YA-equation}
 Y_A''+\frac{n-1}{s}Y_A'+Y_A=-w_A\log w_A^2,
 \qquad Y_A(0)=Y_A'(0)=0.
\end{equation}
The function $t\mapsto t\log t^2$, extended by zero at $t=0$, is uniformly
continuous on $[-1,1]$. Hence
\[
 w_A\log w_A^2\longrightarrow\Phi_n\log\Phi_n^2
 \quad\hbox{uniformly on }[0,S].
\]
The Volterra formula for \eqref{eq:YA-equation}, followed by the same
derivative estimate used above, gives
$Y_A\to\Psi_n$ in $C^1([0,S])$. This proves
\eqref{eq:log-second-profile-convergence}.

We now study the zeros. Every positive zero of $\Phi_n$ is simple.
Indeed, if  $\Phi_n(\rho)=\Phi_n'(\rho)=0$ for some $\rho>0$, then uniqueness for the linear equation \eqref{eq:main-bessel} would imply $\Phi_n\equiv0$, contradicting $\Phi_n(0)=1$. Fix $j\geq1$ and choose $S\in(\rho_j,\rho_{j+1})$.
{The $C^1$ convergence and simple-zero stability} imply that
$w_A$ has exactly one zero near each of $\rho_i$, and no other zeros
in $[0,S]$ for all sufficiently large $A$. Hence $w_A$ has exactly $j$
positive zeros in $[0,S]$, and its $i$-th positive zero converges to
$\rho_i$ for every $1\leq i\leq j$. By the definition of the scaling,
these zeros are $s_i(A):=\sqrt A z_i(\alpha)$. {From $w_A(s_i(A))=0$ and
\eqref{eq:log-second-profile-convergence},
\[
 0=A\Phi_n(s_i(A))+\Psi_n(s_i(A))+o(1).
\]
Because $\Psi_n(s_i(A))$ stays bounded and $\Phi_n'(\rho_i)\ne0$, the
mean-value theorem first gives $s_i(A)-\rho_i=O(A^{-1})$.  A Taylor
expansion of $\Phi_n$ at $\rho_i$, together with
$s_i(A)\to\rho_i$, then yields}
\begin{equation*}
 A(s_i(A)-\rho_i)\longrightarrow
 -\frac{\Psi_n(\rho_i)}{\Phi_n'(\rho_i)}\quad \text{as}~A\rightarrow\infty.
\end{equation*}
To identify the constant, multiply \eqref{eq:log-second-profile} by
$s^{n-1}\Phi_n$ and use the equation for $\Phi_n$. Then
\[
 \bigl[s^{n-1}(\Phi_n\Psi_n'-\Phi_n'\Psi_n)\bigr]'
 =-s^{n-1}\Phi_n^2\log\Phi_n^2.
\]
The regular initial conditions eliminate the contribution at zero, and
evaluation at $\rho_i$ gives
\[
 \rho_i^{n-1}\Phi_n'(\rho_i)\Psi_n(\rho_i)
 =\int_0^{\rho_i}s^{n-1}\Phi_n^2\log\Phi_n^2\,ds.
\]
Finally, the Bessel energy $\frac12(\Phi_n'^2+\Phi_n^2)$ is strictly
decreasing away from the origin. Therefore $|\Phi_n(s)|<1$ for $s>0$,
the last integral is strictly negative, and $c_{i,n}>0$. This proves
\eqref{eq:zeroasymptotic}; in particular, \eqref{eq:z-zero} follows.
\end{proof}

\begin{remark*}
{\rm The regular solution $\Phi_n$ admits the representation
\[
 \Phi_n(s)=2^{\frac n2-1}\Gamma\!\left(\frac n2\right)
 s^{1-\frac n2}J_{\frac n2-1}(s).
\]
Consequently, $\rho_j$ is the $j$-th positive zero of the Bessel function of the first kind $J_{n/2-1}$.}
\end{remark*}

\subsection{The logarithmic finite-ball branch}
We now combine the results established in the preceding subsections to prove Theorem~\ref{thm:log-main}.

\begin{proof}[Proof of Theorem~\ref{thm:log-main}]We prove the theorem in three steps.
\mbox{}\par
\medskip
\noindent\emph{Step 1: Existence.}
Fix $R>0$ and {$k\geq0$}, and set $j=k+1$. By Lemma~\ref{lem:log-left}, the zero curve $z_j(\alpha)$ is defined precisely on $(\alpha_{j-1},\infty)=(\alpha_k,\infty)$, is of class $C^1$, and is strictly decreasing. Moreover, \eqref{eq:z-infty} and \eqref{eq:z-zero} give
 \begin{equation*}
 \lim_{\alpha\downarrow\alpha_k}z_{k+1}(\alpha)=\infty,
 \qquad
 \lim_{\alpha\to\infty}z_{k+1}(\alpha)=0.
 \end{equation*}
Consequently, $z_{k+1}:(\alpha_k,\infty)\rightarrow (0,\infty)$ is a {strictly decreasing continuous bijection}.
Hence there exists a unique $\beta_k(R)>\alpha_k$ such that
 \begin{equation*}
 z_{k+1}(\beta_k(R))=R.
 \end{equation*}
Define
\[
 U_{k,R}(r):=u(r,\beta_k(R)),
 \qquad 0\leq r\leq R.
\]
Then $U_{k,R}(R)=0$. Since $R$ is the $(k+1)$-st positive zero of $u(\cdot,\beta_k(R))$, the interval $(0,R)$ contains exactly $k$ zeros. The interior zeros are simple by Lemma~\ref{lem:log-global}, while its one-sided version shows that the boundary zero is also simple. In particular, $U_{k,R}'(R)\ne0$. Thus $U_{k,R}$ is a radial solution of \eqref{eq:main-log-ball}, positive at the origin, with exactly $k$
simple interior zeros and a simple boundary zero.

\medskip
\noindent\emph{Step 2: Uniqueness.}
Let $U$ be a nontrivial radial solution of \eqref{eq:main-log-ball} with exactly $k$ zeros in $(0,R)$. If
$U(0)=0$, radial regularity gives $U'(0)=0$, and the uniqueness statement in Lemma~\ref{lem:log-global} implies that $U\equiv0$, a contradiction. Hence $U(0)\ne0$. Replacing $U$ by $-U$ if necessary, we may assume that $U(0)=\alpha>0$. Regularity at the origin gives $U'(0)=0$, and uniqueness of the
regular shooting problem yields
\[
 U(r)=u(r,\alpha),
 \qquad 0\leq r\leq R.
\]
By Lemma~\ref{lem:log-global} and its one-sided boundary version, all
zeros of $U$ in $[0,R]$ are simple. Since $U$ has exactly $k$ zeros
in $(0,R)$ and satisfies $U(R)=0$, the point $R$ is the $(k+1)$-st
positive zero of $u(\cdot,\alpha)$. Therefore, $z_{k+1}(\alpha)=R$.
The injectivity of $z_{k+1}$ gives $\alpha=\beta_k(R)$,
and hence $U=U_{k,R}$. Restoring the possible sign change, every
{nontrivial radial solution with exactly $k$ interior zeros is
equal to either $U_{k,R}$ or $-U_{k,R}$.}

\medskip
\noindent\emph{Step 3: Properties.}
{
By Lemma~\ref{lem:log-C1} and Theorem~\ref{thm:log-whole}(v),
{the function $z_{k+1}'$ is continuous and negative}. Differentiating
$z_{k+1}(\beta_k(R))=R$ gives
\[
 \beta_k'(R)=\frac{1}{z_{k+1}'(\beta_k(R))}
 =-\frac{u_r(R,\beta_k(R))}{v(R,\beta_k(R))}<0.
\]
Thus $\beta_k$ is $C^1$ and strictly decreasing. Since
$z_{k+1}:(\alpha_k,\infty)\to(0,\infty)$ is a decreasing bijection,
its inverse has range $(\alpha_k,\infty)$ and satisfies
\[
 \beta_k(R)\to\infty\quad(R\downarrow0),
 \qquad
 \beta_k(R)\to\alpha_k\quad(R\to\infty).
\]
Put $A_R=\log\beta_k^2(R)$, $\rho=\rho_{k+1}$, and
$c=c_{k+1,n}$. Lemma~\ref{lem:log-bessel} gives
\[
 R\sqrt{A_R}=\rho+\frac{c}{A_R}+o(A_R^{-1}).
\]
Since $R^2A_R\to\rho^2$, squaring and rearranging yield
\[
 A_R=\frac{\rho^2}{R^2}+\frac{2c}{\rho}+o(1).
\]
The definitions \eqref{eq:zeroasymptotic} and
\eqref{eq:main-log-kappa} give $2c/\rho=\kappa_{k+1,n}$, proving the
second-order expansion in \eqref{eq:main-log-limits}.
Moreover, Lemma~\ref{lem:log-C1} gives
$u(\cdot,\beta_k(R))\to u(\cdot,\alpha_k)$ in $C^1$ on every compact
interval as $R\to\infty$. The simplicity of the zeros implies
convergence of the corresponding nodal radii. On compact sets that
avoid the limiting nodal set, the logarithmic nonlinearity is smooth,
so the equation upgrades the convergence to $C^2$. This proves
\eqref{eq:main-log-convergence} and completes the proof of
Theorem~\ref{thm:log-main}.}
\end{proof}
\section{The sublinear Dirichlet problem}
\label{sec:proof-sub}
Throughout this section, we write $\alpha_k$ and $\beta_k$ for
$\alpha_k^{\rm sub}$ and $\beta_k^{\rm sub}$, respectively, and set
\[
 F(s)=\int_0^s\bigl(t-|t|^{q-2}t\bigr)\,dt
 =\frac12s^2-\frac1q|s|^q.
\]
The zero-curve properties needed for
Theorem~\ref{thm:sub-main} follow from the shooting theory of
Zhang--Zhang \cite{ZhangZhang}, supplemented by a Bessel-type limit
as the shooting height tends to infinity. We establish these
properties first. The subsequent existence and uniqueness argument
parallels that for the logarithmic problem, whereas the behavior near
the compact-support endpoint requires a separate analysis.

\subsection{Sublinear shooting and zero curves}
For the sublinear nonlinearity $f_{\rm sub}$ in \eqref{eq:intro-sub},
 the regular radial shooting problem is
\begin{equation}
 \begin{cases}
 u''+\dfrac{n-1}{r}u'+u-|u|^{q-2}u=0,\qquad  r>0,\\[4pt]
 u(0)=\alpha>0,\qquad u'(0)=0.
 \end{cases}
 \label{eq:sublinear-shooting}
\end{equation}
We denote by $u(\cdot,\alpha)$ the solution with shooting height $\alpha$,
continued up to its first double zero. Let \(z_u(\alpha)\in(0,\infty]\) denote the location of this first double zero, with the convention that
$z_u(\alpha)=\infty$ if no double zero occurs. Every zero of
$u(\cdot,\alpha)$ in $(0,z_u(\alpha))$ is simple.

Zhang--Zhang \cite[Theorem~1.1]{ZhangZhang} proved that there exists a strictly increasing sequence
\[
 \alpha_0<\alpha_1<\alpha_2<\cdots,\qquad
 \alpha_k\rightarrow\infty,
\]
such that $u(\cdot,\alpha_0)$ is the unique compactly supported ground
state and, for every $k\geq1$, $u(\cdot,\alpha_k)$ is the unique
compactly supported {$k$-node bound state}. We denote the support radius of $u(\cdot,\alpha_k)$ by $S_k$. By
\cite[Remark~2.8]{ZhangZhang}, the solution reaches its first double
zero precisely at the boundary of its support. Hence $S_k=z_u(\alpha_k)$. We also recall that $0<\rho_1<\rho_2<\cdots$ are the positive zeros of
the Bessel profile introduced in \eqref{eq:main-bessel}, and each of these
zeros is simple.

\begin{lemma}
\label{lem:sublinear-zero-curve}
Fix $k\geq0$. For every $\alpha>\alpha_k$, let
$u(\cdot,\alpha)$ be the regular shooting solution of
\eqref{eq:sublinear-shooting}, and let $v(\cdot,\alpha)=\partial_\alpha u(\cdot,\alpha)$ be its shooting variation.
Then the first $k+1$ positive zeros
exist, are all simple, and satisfy
\[
 0<z_1(\alpha)<\cdots<z_{k+1}(\alpha)<z_u(\alpha),
\]
where $z_j(\alpha)$ denotes the $j$-th positive zero of
$u(\cdot,\alpha)$.
For every \(1\leq j\leq k+1\), the zero curve $\alpha\mapsto z_j(\alpha)$ is of class \(C^1\) on \((\alpha_k,\infty)\), and satisfies
\begin{equation}\label{eq:sublinear-zero-derivative-formula}
 z_{j}'(\alpha)
=
 -\frac{v(z_{j}(\alpha),\alpha)}
        {u_r(z_{j}(\alpha),\alpha)}.
\end{equation}
Moreover, the shooting variation satisfies
\begin{equation}
 u_r(z_j(\alpha),\alpha)\,
 v(z_j(\alpha),\alpha)>0,
 \label{eq:sublinear-zero-sign}
\end{equation}
and hence
\begin{equation}
 z_{j}'(\alpha)<0.
 \label{eq:zero-curve-derivative}
\end{equation}
\end{lemma}

\begin{proof}
The classification theorem of Zhang--Zhang \cite{ZhangZhang} gives
\[
 N(\alpha)\geq k+1
 \quad\text{for every }\alpha>\alpha_k,
\]
where \(N(\alpha)\) denotes the number of simple zeros of
\(u(\cdot,\alpha)\) lying before
\(z_u(\alpha)\).  Indeed, if
\(\alpha\in(\alpha_m,\alpha_{m+1})\) for some \(m\geq k\), then
\(N(\alpha)=m+1\), whereas if \(\alpha=\alpha_m\) for some
\(m\geq k+1\), the corresponding bound state has exactly \(m\)
simple zeros.  Hence {the first $k+1$ positive zeros} of
\(u(\cdot,\alpha)\) exist, are all simple, and {satisfy}
$0<z_1(\alpha)<\cdots<z_{k+1}(\alpha)<z_u(\alpha)$.

We next establish the differentiability of the zero curves. The nonlinearity
\[
f(s)=s-|s|^{q-2}s
\]
is of class \(C^1\) away from \(s=0\), with
\[
f'(s)=1-(q-1)|s|^{q-2}.
\]
Let \(z\) be a simple zero of \(u(\cdot,\alpha)\). Then, in a
neighborhood of \(z\),
\[
u(r,\alpha)=(r-z)\eta(r),
\qquad
\eta(z)=u_r(z,\alpha)\neq0.
\]
Consequently,
\[
|u(r,\alpha)|^{q-2}
=
|r-z|^{q-2}|\eta(r)|^{q-2}.
\]
Since \(q-2>-1\), it follows that
\[
f'\bigl(u(\cdot,\alpha)\bigr)\in L^1_{\mathrm{loc}}
\]
near every simple zero. The \(L^1\)-coefficient Cauchy theory of
\cite[Proposition~2.1]{ZhangZhang}, together with
\cite[Lemmas~2.2 and 3.1]{ZhangZhang},
therefore shows that the shooting solution $u(\cdot,\alpha)$ {depends $C^1$ on $\alpha$} on every compact interval lying before \(z_u(\alpha)\). Its shooting variation $v(\cdot,\alpha)$
is the unique regular solution, understood across the simple zeros in
the \(L^1\)-coefficient sense, of
\[
 v''+\frac{n-1}{r}v'+\bigl(1-(q-1)|u|^{q-2}\bigr)v=0,
 \qquad
 v(0)=1,
 \qquad
 v'(0)=0.
\]
Fix \(1\leq j\leq k+1\). Since $u_r\bigl(z_j(\alpha),\alpha\bigr)\neq0$,
the implicit function theorem, together with the {continuous dependence of $u(\cdot,\alpha)$ on $\alpha$ in $C^1$}
gives $z_j\in C^1\bigl((\alpha_k,\infty)\bigr)$.
Differentiating $u\bigl(z_j(\alpha),\alpha\bigr)=0$
with respect to \(\alpha\), we obtain \eqref{eq:sublinear-zero-derivative-formula}.

{Zhang--Zhang \cite[Lemma 5.1]{ZhangZhang} gives the required
zero distribution of the variation (with its ground-state and bound-state
cases treated separately there):}
\(v(\cdot,\alpha)\) has exactly one zero in each interval
\((z_{j-1}(\alpha),z_j(\alpha))\), $1\leq j \leq k+1$, where $z_0(\alpha):=0$. Since
\[
u(0,\alpha)>0,
\qquad
v(0,\alpha)=1>0,
\]
the resulting interlacing pattern yields
\[
\operatorname{sgn}u_r\bigl(z_j(\alpha),\alpha\bigr)
=
\operatorname{sgn}v\bigl(z_j(\alpha),\alpha\bigr)
=
(-1)^j,
\qquad 1\leq j\leq k+1.
\]
 Hence \eqref{eq:sublinear-zero-sign} holds. Combining
\eqref{eq:sublinear-zero-derivative-formula} with
\eqref{eq:sublinear-zero-sign}, we obtain \eqref{eq:zero-curve-derivative}.
\end{proof}

\subsection{Endpoint behavior of zero curves}
We next analyze the zero curves at the two endpoints of their domains.
The behavior as \(\alpha\downarrow\alpha_k\) is determined by the
compactly supported \(k\)-node bound state, whereas the regime
\(\alpha\to\infty\) is governed by a Bessel scaling limit. These two
limits are established in the following lemmas.
\begin{lemma}

For every \(k\geq0\), the zero curve \(z_{k+1}\) admits a continuous
extension to the left endpoint \(\alpha=\alpha_k\), given by $z_{k+1}(\alpha_k):=S_k$. Equivalently,
\begin{equation}
 z_{k+1}(\alpha)\uparrow S_k\quad \text{as }\alpha\downarrow\alpha_k,
 \label{eq:sublinear-left-limit}
\end{equation}
where \(S_k\) denotes the support radius of the compactly supported
\(k\)-node bound state $u(\cdot,\alpha_k)$.
\end{lemma}

\begin{proof}
Set
\[
u_k:=u(\cdot,\alpha_k),
\qquad
v_k:=v(\cdot,\alpha_k),
\qquad
\sigma_k:=(-1)^k.
\]
The compactly supported bound state \(u_k\) has exactly \(k\)
simple zeros in \((0,S_k)\) and satisfies $u_k(S_k)=u_k'(S_k)=0$.
Since \(u_k(0)>0\), its sign on the terminal nodal interval is
\((-1)^k\). Hence $\sigma_k u_k(r)>0$
to the right of its last simple zero.
{The ground-state and bound-state alternatives in
\cite[Lemma~5.1(b),(c)]{ZhangZhang} show that $v_k$ has a last zero before
$S_k$, is eventually monotone, and satisfies
$|v_k(r)|\to\infty$ as $r\uparrow S_k$.}
Together with the phase-transition sign pattern, this implies that,
sufficiently close to \(S_k\),
\[
\sigma_k u_k(r)>0,
\qquad\sigma_k u_k'(r)<0,
\qquad\sigma_k v_k(r)<0,
\qquad
\sigma_k v_k'(r)<0.
\]
Choose \(a>0\) so small that
\[
f'(s)=1-(q-1)s^{q-2}<0
\quad\text{for }0<s<a.
\]
{Fix $\varepsilon\in(0,S_k)$.} We may then choose
$R\in(S_k-\varepsilon,S_k)$
to the right of all the simple zeros of \(u_k\) and sufficiently
close to \(S_k\) so that
\begin{equation}\label{inequa}
  0<\sigma_k u_k(r)<a,
\qquad
\sigma_k u_k'(r)<0,
\qquad
\sigma_k v_k(r)<0,
\qquad
\sigma_k v_k'(r)<0\quad\text{for }R\leq r<S_k.
\end{equation}

We give the terminal jumping argument.  By differentiability with respect to
the shooting height on the fixed interval $[0,R]$ and the last two signs in
\eqref{inequa}, there is $\delta_1>0$ such that, for
$\alpha_k<\beta<\alpha_k+\delta_1$,
\begin{equation}\label{eq:terminal-initial-difference}
 d_\beta(R)<0,
 \qquad d_\beta'(R)<0,
 \qquad
 d_\beta:=\sigma_k\bigl(u(\cdot,\beta)-u_k\bigr).
\end{equation}
Suppose that $u(\cdot,\beta)$ has no zero in $(R,S_k)$.  As long as
$0<\sigma_k u(r,\beta)<a$, the difference satisfies
\[
 (r^{n-1}d_\beta')'+r^{n-1}c_\beta(r)d_\beta=0,
 \qquad
 c_\beta(r)=
 \frac{f(\sigma_k u(r,\beta))-f(\sigma_k u_k(r))}
 {\sigma_k u(r,\beta)-\sigma_k u_k(r)}<0,
\]
where oddness of $f$ has been used.  Starting from
\eqref{eq:terminal-initial-difference}, a first-contact argument gives
$d_\beta<0$ and $d_\beta'<0$ on $[R,S_k]$: indeed, while
$d_\beta<0$, one has
$(r^{n-1}d_\beta')'=-r^{n-1}c_\beta d_\beta<0$, so the derivative
cannot vanish.  In particular,
$0<\sigma_k u(r,\beta)<\sigma_k u_k(r)<a$, and the comparison remains
valid up to $S_k$.  But then
\[
 0\leq \sigma_k u(S_k,\beta)
 =d_\beta(S_k)<d_\beta(R)<0,
\]
a contradiction.  Hence $u(\cdot,\beta)$ has a zero in $(R,S_k)$.
It is simple by Lemma~\ref{lem:sublinear-zero-curve}, which applies
because $\beta>\alpha_k$.
It remains to identify this zero relative to the preceding zeros.
By the {continuous dependence of the shooting solution on the
initial height in $C^1$},
\[
u(\cdot,\beta)\rightarrow u_k
\quad{\text{in }C^1([0,R])\quad\text{as }\beta\downarrow\alpha_k.}
\]
Since the \(k\) zeros of \(u_k\) in \((0,R)\) are simple, each of
them persists as a unique simple zero of \(u(\cdot,\beta)\) for
\(\beta>\alpha_k\) sufficiently close to \(\alpha_k\). Furthermore,
by choosing disjoint neighborhoods of these zeros and using uniform
convergence on their complement, we see that no additional zeros
occur in \([0,R]\). Thus there exists \(\delta_2>0\) such that
\(u(\cdot,\beta)\) has exactly \(k\) zeros in \((0,R)\) whenever $\alpha_k<\beta<\alpha_k+\delta_2$.
Consequently, for $\alpha_k<\beta<\alpha_k+\min\{\delta_1,\delta_2\}$,
the first zero of \(u(\cdot,\beta)\) lying to the right of \(R\) is
its \((k+1)\)-st positive zero. Since at least one such zero lies in
\((R,S_k)\), we obtain
\[
R<z_{k+1}(\beta)<S_k.
\]
It follows from \(R>S_k-\varepsilon\) that
\[
S_k-\varepsilon
<
z_{k+1}(\beta)
<
S_k
\]
for every \(\beta>\alpha_k\) sufficiently close to \(\alpha_k\).
Since \(\varepsilon>0\) is arbitrary, we conclude that
\[
z_{k+1}(\beta)\uparrow S_k\quad \text{as }\beta\downarrow\alpha_k,
\]
which is \eqref{eq:sublinear-left-limit}.
\end{proof}

\begin{lemma}
\label{lem:sublinear-bessel-limit}
For every fixed $j\ge1$, the zero curve \(z_j\), defined on
\((\alpha_{j-1},\infty)\), satisfies
\begin{equation}
 z_j(\alpha)\downarrow\rho_j
 \quad\text{as }\alpha\to\infty.
 \label{eq:sublinear-right-limit-general}
\end{equation}
More precisely,
\begin{equation}
 \lim_{\alpha\to\infty}\alpha^{2-q}\bigl(z_j(\alpha)-\rho_j\bigr)
 = \mathfrak c_{j,q,n},
 \label{eq:sublinear-first-order-zero}
\end{equation}
where $\mathfrak c_{j,q,n}$ is defined in
\eqref{eq:sub-constant-intro}.
\end{lemma}

\begin{proof}
For $\alpha>1$, set
\[
 w_\alpha(r):=\frac{u(r,\alpha)}{\alpha},
 \qquad
 \varepsilon_\alpha:=\alpha^{q-2}.
\]
Since $1<q<2$, we have $\varepsilon_\alpha\rightarrow0$ as $\alpha\to\infty$.
If the shooting solution reaches a first double zero, we use its zero extension beyond that point. Thus $w_\alpha$ satisfies
\begin{equation}
 \begin{cases}
 w_\alpha''+\dfrac{n-1}{r}w_\alpha'+w_\alpha
 -\varepsilon_\alpha|w_\alpha|^{q-2}w_\alpha=0,\qquad  r>0,\\[4pt]
 w_\alpha(0)=1,\qquad w_\alpha'(0)=0.
 \end{cases}
 \label{eq:sublinear-scaled}
\end{equation}
  The energy inequality gives
$|u(r,\alpha)|<\alpha$ before the first double zero. Hence, with the zero extension,
\begin{equation}
|w_\alpha(r)|\leq1
\quad\text{for all }r\geq0.
\label{eq:upper1}
\end{equation}
The Volterra representation of \eqref{eq:sublinear-scaled} is
\begin{align}
 w_\alpha(r)=1-
 \int_0^r t^{1-n}\int_0^t s^{n-1}
 \Bigl(w_\alpha(s)-\varepsilon_\alpha
 |w_\alpha(s)|^{q-2}w_\alpha(s)\Bigr)\,ds\,dt.
 \label{eq:sublinear-volterra}
\end{align}
In particular,
\begin{equation*}
w_\alpha'(r)=
-r^{1-n}\int_0^r s^{n-1}
\Bigl(
w_\alpha(s)
-\varepsilon_\alpha
|w_\alpha(s)|^{q-2}w_\alpha(s)
\Bigr)ds.
\end{equation*}
Since {$|w_\alpha|\leq1$} and $\varepsilon_\alpha\rightarrow0$,
\begin{equation*}
\big|w_\alpha-\varepsilon_\alpha
|w_\alpha|^{q-2}w_\alpha\big|\leq |w_\alpha|+\varepsilon_\alpha|w_\alpha|^{q-1}\leq 1+\varepsilon_\alpha\leq2,
\end{equation*}
and therefore,
\begin{equation*}
|w_\alpha'(r)|\leq
2r^{1-n}\int_0^r s^{n-1}ds\leq \frac{2}{n}r.
\end{equation*}
{For $r>0$, the equation also gives
\[
 |w_\alpha''(r)|
 \le \frac{n-1}{r}|w_\alpha'(r)|
 +\bigl|w_\alpha(r)-\varepsilon_\alpha
 |w_\alpha(r)|^{q-2}w_\alpha(r)\bigr|
 \le \frac{2(n-1)}{n}+2.
\]
The same estimate at $r=0$ follows from the regular integral
formulation. Hence $\{w_\alpha'\}$ is equi-Lipschitz on every fixed
compact interval.}
Hence, for every
$L>0$, the families $\{w_\alpha\}$ and $\{w_\alpha'\}$ are uniformly
bounded and equicontinuous on $[0,L]$. The Arzel\`{a}--Ascoli theorem
therefore yields relative compactness of $\{w_\alpha\}$ in
$C^1([0,L])$. Let $\tilde w$ be the $C^1$-limit of
a subsequence. Passing to the limit in \eqref{eq:sublinear-volterra}, and using
$\varepsilon_{\alpha}\to0$, we find that $\tilde w$ solves \eqref{eq:main-bessel}. By the uniqueness of the regular solution, $\tilde w=\Phi_n$. Since every convergent subsequence has the same limit, the whole family satisfies
\begin{equation}\label{eq:y-to-Phi}
w_\alpha\rightarrow\Phi_n\quad \text{in }C^1_{\rm loc}([0,\infty)){\quad\text{as }\alpha\to\infty}.
\end{equation}
By the $C^1_{\mathrm{loc}}$ convergence in \eqref{eq:y-to-Phi}, the simplicity of the zeros of $\Phi_n$, and the argument used in the proof of Lemma~\ref{lem:log-bessel}, we obtain
 $z_j(\alpha)\rightarrow\rho_j$ as $\alpha\to\infty$. Since $z_j$ is strictly decreasing in $\alpha$ by \eqref{eq:zero-curve-derivative}, it follows that $z_j(\alpha)\downarrow\rho_j$ as $\alpha\to\infty$,
which proves \eqref{eq:sublinear-right-limit-general}.

We now retain the first perturbative term.  Define
\[
 y_\alpha:=\frac{w_\alpha-\Phi_n}{\varepsilon_\alpha}.
\]
Subtracting the equation for $\Phi_n$ from
\eqref{eq:sublinear-scaled} and dividing by
$\varepsilon_\alpha$ gives
\begin{equation}
 \begin{cases}
 y_\alpha''+\dfrac{n-1}{r}y_\alpha'+y_\alpha
 =|w_\alpha|^{q-2}w_\alpha,\qquad  r>0,\\[4pt]
 y_\alpha(0)=y_\alpha'(0)=0.
 \end{cases}
 \label{eq:y-alpha}
\end{equation}
The function $g(s)=|s|^{q-2}s$ is continuous, and hence uniformly continuous, on $[-1,1]$.
Consequently, \eqref{eq:upper1} and \eqref{eq:y-to-Phi} imply
$g(w_\alpha)\to g(\Phi_n)$ uniformly on every compact interval.
The Volterra formulation of \eqref{eq:y-alpha}, together with the
standard continuous-dependence estimate for the regular linear
equation, therefore gives
\begin{equation*}
 y_\alpha\rightarrow\Psi_n
 \quad\text{in }C^1_{\rm loc}([0,\infty)){\quad\text{as }\alpha\rightarrow\infty},
\end{equation*}
where $\Psi_n$ is the unique regular solution of
\begin{equation}
 \begin{cases}
 \Psi''+\dfrac{n-1}{r}\Psi'+\Psi
 =|\Phi_n|^{q-2}\Phi_n,\qquad  r>0,\\[4pt]
 \Psi(0)=\Psi'(0)=0.
 \end{cases}
 \label{eq:Psi-equation}
\end{equation}
For brevity, write $z_j=z_j(\alpha)$ and $\rho=\rho_j$.
{Since $w_\alpha(z_j)=0$,
\[
 \Phi_n(z_j)+\varepsilon_\alpha y_\alpha(z_j)=0.
\]
The local uniform boundedness of $y_\alpha$ gives
$|\Phi_n(z_j)|=O(\varepsilon_\alpha)$.  Because $\rho$ is a simple zero
and $z_j\to\rho$, this first implies
$z_j-\rho=O(\varepsilon_\alpha)$.  Applying the mean-value theorem and
using $y_\alpha(z_j)\to\Psi_n(\rho)$ then gives}
\begin{equation}
\frac{z_j-\rho}{\varepsilon_\alpha}
\longrightarrow
-\frac{\Psi_n(\rho)}{\Phi_n'(\rho)}\quad \text{as }\alpha\rightarrow\infty.
\label{eq:zero-shift-Psi}
\end{equation}

It remains to compute the constant on the right-hand side of \eqref{eq:zero-shift-Psi}. Multiply \eqref{eq:Psi-equation} by $\Phi_n$, multiply the equation for $\Phi_n$ by $\Psi_n$, and subtract. This gives %the Wronskian identity
\[
 \Bigl[r^{n-1}(\Phi_n\Psi_n'-\Phi_n'\Psi_n)\Bigr]'
 =r^{n-1}|\Phi_n|^q.
\]
{Integrating over $(0,\rho)$, the contribution at the origin
vanishes by regularity, and $\Phi_n(\rho)=0$. Hence}
\[
 -\rho^{n-1}\Phi_n'(\rho)\Psi_n(\rho)
 =\int_0^\rho r^{n-1}|\Phi_n(r)|^q\,dr.
\]
Since $\Phi_n'(\rho)\neq0$, it follows that
\[
 -\frac{\Psi_n(\rho)}{\Phi_n'(\rho)}
 =\frac{\displaystyle\int_0^\rho
 r^{n-1}|\Phi_n(r)|^q\,dr}
 {\rho^{n-1}\Phi_n'^2(\rho)}
 =\mathfrak c_{j,q,n}>0.
\]
Since $\varepsilon_\alpha=\alpha^{q-2}$,
\eqref{eq:zero-shift-Psi} is precisely
\eqref{eq:sublinear-first-order-zero}.
\end{proof}

\subsection{\texorpdfstring{{Convergence and free-boundary asymptotics at the compact-support endpoint}}{Convergence and free-boundary asymptotics at the compact-support endpoint}}

By Lemmas~\ref{lem:sublinear-zero-curve}--\ref{lem:sublinear-bessel-limit}, the zero curve defines a strictly
decreasing $C^1$ bijection
\[
z_{k+1}:(\alpha_k,\infty)
\longrightarrow(\rho_{k+1},S_k).
\]
For $R\in(\rho_{k+1},S_k)$, we denote its inverse by
$\beta_k(R)$.  {We first prove convergence of the stopped profiles on the
fixed limiting support and then record the free-boundary rates of the endpoint
state.}

\begin{lemma}
\label{lem:sublinear-state-convergence}
Let $R\uparrow S_k$ and set $\beta_R:=\beta_k(R)$,
$U_{k,R}:=u(\cdot,\beta_R)$. Define $\widetilde U_{k,R}$ to be the zero extension of the
{radial profile} $U_{k,R}$ from {$[0,R]$ to $[0,S_k]$}.
Then
\begin{equation}
 \widetilde U_{k,R}\longrightarrow u_k
 \quad\text{in }W^{1,\infty}(0,S_k),
 \label{eq:sub-W1infty}
\end{equation}
where $u_k=u(\cdot,\alpha_k)$ denotes the compactly supported
\(k\)-node bound state.
Moreover, the convergence is locally $C^2$ away from the
simple zeros of $u_k$.
\end{lemma}

\begin{proof}
Since $z_{k+1}(\beta_R)=R$ and
$R\uparrow S_k$, the inverse zero-curve limit gives $\beta_R\downarrow\alpha_k$.
Consequently, for every fixed $r_0<S_k$, the continuous-dependence
results of Zhang--Zhang \cite{ZhangZhang} yield
\begin{equation}
u(\cdot,\beta_R)\rightarrow u_k
\quad\text{in }C^1([0,r_0]).
\label{eq:sub-local-C1}
\end{equation}
On every compact interval that contains no zero of $u_k$, the
nonlinearity is smooth. The equation therefore upgrades
\eqref{eq:sub-local-C1} to local $C^2$ convergence.

It remains to control the terminal interval near $S_k$. Write
\[
E_\gamma(r)
:=
\frac12u_r^2(r,\gamma)+F(u(r,\gamma)).
\]
Since $u_k(S_k)=u_k'(S_k)=0$,
we have $E_{\alpha_k}(r)\rightarrow0$ as $r\uparrow S_k$.
{Choose $r_0<S_k$ as follows.  If $k=0$, the ground state is
strictly decreasing by \cite[Proposition~2.5(ii)]{ZhangZhang}, and the nearby
shooting profiles are also strictly decreasing up to their first zero.  Indeed,
the energy at that simple zero is positive, so $f(u(0))>0$ and $u'<0$ just to
the right of the origin; at any preceding critical point one likewise has
$F(u)>0$ and hence $u''=-f(u)<0$.  If $r_*$ were the first such point, however,
$u'<0$ on $(0,r_*)$ would force $u''(r_*)\ge0$, a contradiction.  If $k\ge1$,
choose $r_0$ to the right of the unique critical point in the terminal nodal
interval of $u_k$.
By \cite[Proposition~2.5(iii)]{ZhangZhang}, the corresponding shooting
profile has exactly one critical point between its last interior zero and its
boundary zero.  That critical point persists to the left of $r_0$ under the
local $C^1$ convergence in \eqref{eq:sub-local-C1}.  Thus, in either case,
for all $R$ sufficiently close to $S_k$, $u(\cdot,\beta_R)$ has no critical
point on $[r_0,R]$ and approaches its boundary zero monotonically.} Hence
\begin{equation}
 \sup_{r_0\le r\le R}|u(r,\beta_R)|
 \le |u(r_0,\beta_R)|.
 \label{eq:tail-u-control}
\end{equation}
Since $E_{\beta_R}'=-(n-1)r^{-1}u_r^2(r,\beta_R)\le0$, it follows from \eqref{eq:tail-u-control} that, for $r\in [r_0,R]$,
\begin{align*}
 \frac12u_r^2(r,\beta_R)
 =E_{\beta_R}(r)-F(u(r,\beta_R))
 \le E_{\beta_R}(r_0)
 +\sup_{|s|\le |u(r_0,\beta_R)|}|F(s)|.
\end{align*}
By \eqref{eq:sub-local-C1}, $E_{\beta_R}(r_0)\rightarrow E_{\alpha_k}(r_0)$ and $u(r_0,\beta_R)\rightarrow u_k(r_0)$ as $R\uparrow S_k$.
Therefore,
\begin{equation}
\limsup_{R\uparrow S_k}
\sup_{r_0\leq r\leq R}
u_r^2(r,\beta_R)
\leq
2E_{\alpha_k}(r_0)
+2\sup_{|s|\le |u_k(r_0)|}|F(s)|,
\label{eq:tail-derivative-limsup}
\end{equation}
while \eqref{eq:tail-u-control} gives
\begin{equation}
\limsup_{R\uparrow S_k}
\sup_{r_0\leq r\leq R}
|u(r,\beta_R)|
\leq |u_k(r_0)|.
\label{eq:tail-function-limsup}
\end{equation}
The same energy estimate applies to $u_k$ on $[r_0,S_k]$.
Since $u_k(r)\rightarrow0$ and $E_{\alpha_k}(r)\rightarrow0$ as $r\uparrow S_k$, the right-hand sides of \eqref{eq:tail-derivative-limsup} and \eqref{eq:tail-function-limsup} can be made arbitrarily small by choosing $r_0$ sufficiently close to $S_k$. Since $U_{k,R}(R)=0$, its zero extension
$\widetilde U_{k,R}$ is continuous and piecewise $C^1$ on
$[0,S_k]$, and hence is Lipschitz.
Therefore,
$\widetilde U_{k,R}\in W^{1,\infty}(0,S_k)$. We now prove the convergence.
Choose $r_0<S_k$ sufficiently close
to $S_k$ so that the corresponding tail bounds \eqref{eq:tail-derivative-limsup} and \eqref{eq:tail-function-limsup} for $u(\cdot,\beta_R)$ and $u_k$ are arbitrarily small. The local convergence \eqref{eq:sub-local-C1} controls the interval $[0,r_0]$,
while the preceding tail estimates control the differences of the
functions and their derivatives on $[r_0,R]$. On $[R,S_k]$, since $\widetilde U_{k,R}=\widetilde U_{k,R}'=0$ almost everywhere,
only the corresponding tail bounds for $u_k$ are needed. Combining
these estimates and then letting $R\uparrow S_k$ yields
{$\|\widetilde U_{k,R}-u_k\|_{W^{1,\infty}(0,S_k)}\rightarrow0$},
which proves \eqref{eq:sub-W1infty}.
\end{proof}

\begin{lemma}\label{lem:sub-free-rate}
Let $u_k=u_k(\cdot,\alpha_k)$ and $\sigma_k=(-1)^k$. On the terminal nodal
domain, put $w=\sigma_k u_k$. Then $w>0$, $w'<0$, and the three
asymptotic formulas \eqref{eq:main-free-u}--\eqref{eq:main-free-ddu}
hold as $r\uparrow S_k$.
\end{lemma}

\begin{proof}
{The profile asymptotic in \eqref{eq:main-free-u} and the equivalent
first-derivative ratio are already contained in
\cite[Proposition~2.5(iv)]{ZhangZhang}.  For completeness, and to obtain the
second-derivative rate in the same notation, we give a short energy
derivation of all three formulas.}
On the terminal nodal interval,
\[
 w''+\frac{n-1}{r}w'+w-w^{q-1}=0,
 \qquad w(S_k)=w'(S_k)=0.
\]
Its energy
\[
 E(r)=\frac12w'^2(r)+\frac12w^2(r)-\frac1q w^q(r)
\]
{is strictly positive for $r<S_k$, tends to zero at $S_k$, and satisfies
$E'=-(n-1)r^{-1}w'^2$.  Indeed, $E$ is nonincreasing in $r$ and
$E(S_k)=0$; equality at an earlier point would force $w'\equiv0$ up to the
free boundary and hence contradict $w>0$ on the terminal nodal interval.} Since $w$ is strictly decreasing, it may be
used as a variable near the free boundary. Writing $r=r(w)$ gives
\begin{equation*}
 \frac{dE}{dw}=\frac{n-1}{r(w)}
 \sqrt{2\left(E-\frac12w^2+\frac1q w^q\right)}.
\end{equation*}
After restricting to $r\ge S_k/2$, integration from $0$ to $w$ and
monotonicity of $E$ imply
\[
 E(w)\le Cw\sqrt{E(w)+w^q}.
\]
Solving this quadratic inequality for $\sqrt{E(w)}$ yields
\[
 E(w)=O\big(w^2+w^{1+\frac{q}{2}}\big)=o(w^q),
\]
because $1<q<2$. Therefore
\begin{equation}\label{eq:free-first-integral}
 -w'(r)=\sqrt{\frac2q}\,w(r)^{\frac{q}{2}}(1+o(1)).
\end{equation}
Integrating \eqref{eq:free-first-integral} up to $S_k$ gives
\[
 w(r)^{\frac{2-q}{2}}
 =\frac{2-q}{\sqrt{2q}}(S_k-r)(1+o(1)),
\]
which is \eqref{eq:main-free-u}; substitution back into
\eqref{eq:free-first-integral} gives \eqref{eq:main-free-du}. Finally,
the differential equation gives
\[
 w''=w^{q-1}-w-\frac{n-1}{r}w'=w^{q-1}(1+o(1)),
\]
since $|w'|=O(w^{q/2})=o(w^{q-1})$ and $w=o(w^{q-1})$. This proves
\eqref{eq:main-free-ddu}.
\end{proof}

\subsection{The maximal sublinear simple-zero branch}
We now combine the results established in the preceding subsections to prove Theorem~\ref{thm:sub-main}.
\begin{proof}[Proof of Theorem~\ref{thm:sub-main}]
{
By Lemmas~\ref{lem:sublinear-zero-curve}--\ref{lem:sublinear-bessel-limit},
\[
 z_{k+1}:(\alpha_k,\infty)\longrightarrow
 (\rho_{k+1},S_k)
\]
is a strictly decreasing $C^1$ bijection.  In particular,
$S_k>\rho_{k+1}$.  For every
$R\in(\rho_{k+1},S_k)$, there is therefore a unique
$\beta_k(R)>\alpha_k$ such that
\[
 z_{k+1}(\beta_k(R))=R.
\]
Set $U_{k,R}(r):=u(r,\beta_k(R))$ for $0\le r\le R$.
Then $U_{k,R}$ solves \eqref{eq:main-sub-ball}, is positive at the
origin, and has exactly $k$ simple zeros in $(0,R)$.  Its boundary
zero is also simple.  Indeed, if
$U_{k,R}(R)=U_{k,R}'(R)=0$, its zero extension beyond $R$ would be a
compactly supported $k$-node whole-space {bound state}.  The uniqueness
theorem of Zhang--Zhang \cite{ZhangZhang} would then force
$R=S_k$, contrary to $R<S_k$.

We next prove both uniqueness and the necessity of the radius range.
Let $U$ be any radial solution in the simple-zero class stated in the
theorem.  Since a nontrivial regular radial solution cannot have
$U(0)=U'(0)=0$, we have $U(0)\ne0$; {because the equation is invariant under $U\mapsto -U$}, we may assume
$\alpha:=U(0)>0$.  The nonlinear Cauchy problem is unique up to the
first double zero.  Because every zero of $U$ on $[0,R]$ is simple,
this uniqueness can be continued successively across the nodal set,
and hence
\[
 U(r)=u(r,\alpha),\qquad 0\le r\le R.
\]
The point $R$ is the $(k+1)$-st simple zero of this shooting orbit.
The sublinear shooting classification therefore gives
$\alpha>\alpha_k$ and
\[
 R=z_{k+1}(\alpha)\in(\rho_{k+1},S_k).
\]
Thus no {solution} in this class exists outside the asserted radius
interval.  For $R$ inside the interval, injectivity of $z_{k+1}$ gives
$\alpha=\beta_k(R)$, so $U=U_{k,R}$.  {Restoring the possible sign change proves uniqueness up to sign.}

The inverse-function theorem {applied to the strictly transversal zero
curve shows that $\beta_k\in C^1$ and} gives
\[
 \beta_k'(R)=\frac{1}{z_{k+1}'(\beta_k(R))}
 =-\frac{u_r(R,\beta_k(R))}{v(R,\beta_k(R))}<0.
\]
Inverting the endpoint limits of $z_{k+1}$ in \eqref{eq:sublinear-left-limit} and \eqref{eq:sublinear-right-limit-general} yields
\[
 \beta_k(R)\downarrow\alpha_k\quad(R\uparrow S_k),
 \qquad
 \beta_k(R)\to\infty\quad(R\downarrow\rho_{k+1}).
\]
Moreover, \eqref{eq:sublinear-first-order-zero}, with
$\alpha=\beta_k(R)$ and $z_{k+1}(\alpha)=R$, gives
\[
 (R-\rho_{k+1})\beta_k(R)^{2-q}
 \longrightarrow\mathfrak c_{k+1,q,n}
 \quad(R\downarrow\rho_{k+1}).
\]
This proves \eqref{eq:main-sub-limits}.  Finally,
\eqref{eq:main-sub-convergence} follows from
Lemma~\ref{lem:sublinear-state-convergence}, while
\eqref{eq:main-free-u}--\eqref{eq:main-free-ddu} follow from
Lemma~\ref{lem:sub-free-rate}.
}
\end{proof}

\section{Unified parameter branches on the unit ball}
\label{sec:parameter-branches}

The two radius-dependent families become branches of the single parameter
problem \eqref{eq:unit-parameter} after the change of variables
$x=Ry$ and $\lambda=R^2$. {Although this rescaling is elementary, the
resulting formulation makes the no-fold global topology explicit while
keeping the all-radius logarithmic branch distinct from the maximal sublinear
simple-zero branch.}

\begin{proof}[Proof of Theorem~\ref{thm:branch-main}]
The map $R\mapsto\lambda=R^2$ is an increasing $C^1$ diffeomorphism of
$(0,\infty)$. Theorem~\ref{thm:log-main} therefore gives the logarithmic
domain, range, endpoint expansions, and convergence after replacing $R$ by
$\sqrt\lambda$. Likewise, Theorem~\ref{thm:sub-main} maps
$(\rho_{k+1},S_k)$ onto $(\rho_{k+1}^2,S_k^2)$. Since
\[
 \lambda-\rho_{k+1}^2
 =(\sqrt\lambda-\rho_{k+1})(\sqrt\lambda+\rho_{k+1}),
\]
the first sublinear endpoint law acquires the factor $2\rho_{k+1}$.
The chain rule gives \eqref{eq:branch-derivative}. {Thus the central
height is a valid global coordinate on each stated parameter interval, so no
fold occurs in the shooting-height parametrization.} The limiting double zero
and its profile are exactly Lemma~\ref{lem:sub-free-rate}; no further
continuation is used. This proves the theorem.
\end{proof}

\section{Radial and angular spectra}
\label{sec:proof-morse}
{Sections~\ref{sec:proof-log} and~\ref{sec:proof-sub} established the
two finite-ball classifications, while Section~\ref{sec:parameter-branches}
recast them as monotone unit-ball parameter branches.  We now turn to the
spectral conclusions.  Differentiation with respect to the shooting height
produces the regular zero-energy solution of the radial linearized equation,
and differentiation of the zero-curve identity fixes its boundary sign.  A
singular radial Sturm theorem then gives radial nondegeneracy and the radial
Morse index.  Finally, the degree-one translation mode $U'$ and
spherical-harmonic decomposition provide the angular counts and the full
Morse-index formula.}
\subsection{A zero-curve Sturm principle}
We first formulate the abstract shooting principle used for all three
models. Consider the regular radial problem \eqref{eq:intro-shooting-unified},
let $U(r):=u(r,\alpha)$
be a stopped shooting orbit whose boundary value satisfies $U(R)=0$.
The radial linearized operator at $U$ is
\[
L_U:=-\Delta-f'(U).
\]
{When $L_U$ is applied below to a function that does not satisfy the
Dirichlet boundary condition, it denotes the associated differential expression.}
Write
\begin{equation*}
        a(r):=f'(U(r)).
\end{equation*}
Its radial quadratic form is
\begin{equation}
\label{eq:radial-quadratic-form}
\mathcal Q_U(\phi)
:=
\int_0^R
\left(
|\phi'(r)|^2-a(r)\phi^2(r)
\right)r^{n-1}\,dr,
\qquad
\phi\in H^1_{0,\mathrm{rad}}(B_R).
\end{equation}
For the power nonlinearity, the coefficient $a$ is continuous. In the
logarithmic and sublinear cases, it is singular at the nodal radii of
$U$, but these singularities are locally integrable. As the stopped profile has
only finitely many simple zeros in $(0,R)$, it follows that $a\in L^1(0,R)$.
{Moreover, $a$ is bounded in a neighborhood of $r=0$, because $U(0)>0$ and $f'$ is continuous near $U(0)$.}
Accordingly, \eqref{eq:radial-quadratic-form} is understood in the
usual sense in the power case and as the corresponding closed
lower-semibounded radial form in the logarithmic and sublinear cases.
The radial eigenvalue equation $L_U\phi=\lambda\phi$ takes the form
\begin{equation}\label{eq:eigen-lambda}
        \phi''+\frac{n-1}{r}\phi'+(a(r)+\lambda)\phi=0,
\end{equation}
with the regularity and Dirichlet conditions
\begin{equation*}
\phi'(0)=0,
\qquad
\phi(R)=0.
\end{equation*}
At the singular nodal radii, \eqref{eq:eigen-lambda} is understood in
the integral sense. The form domain throughout this section is
$H^1_{0,\mathrm{rad}}(B_R)$.
In one-dimensional radial coordinates, its smooth radial core consists of
functions satisfying the regular condition at the origin and the Dirichlet
condition at $R$.  {Rather than introducing an auxiliary boundary
condition at $r=\varepsilon$, we will preserve the regular origin condition
and approximate only the singular potential by bounded coefficients.}
We next record the radial Sturm count in the precise form needed to
relate the zeros of the shooting variation to the radial kernel and the Morse index of $U$.

\begin{lemma}\label{lem:singular-form}
Let $a=a_+-a_-$ on $(0,R)$, where $a_+\in L^\infty(0,R)$ and
$a_-\geq0$.  Assume that $a_-$ is bounded near $r=0$ and outside
neighborhoods of finitely many points in $(0,R]$.  Each singularity,
including a possible singularity at $r=R$, is assumed to belong to
$L^1$ on a one-sided neighborhood.  Then
\[
 \phi\longmapsto\int_0^R a_-(r)\phi^2(r)r^{n-1}\,dr
\]
is infinitesimally form bounded with respect to the radial Dirichlet
energy.  Consequently, the form in
\eqref{eq:radial-quadratic-form}, with domain
$H^1_{0,{\rm rad}}(B_R)$, is densely defined, closed, and bounded from
below.  Its self-adjoint realization has compact resolvent.
\end{lemma}

\begin{proof}
Choose pairwise disjoint intervals $I_1,\ldots,I_m$ around the
singular points.  If the last singular point is $R$, take
$I_m=(R-\delta,R)$.  On every $I_j$, the weight $r^{n-1}$ is bounded
above and below by positive constants.  For every $\eta>0$, the
one-dimensional inequality
\[
 \|\phi\|_{L^\infty(I_j)}^2
 \leq \eta\|\phi'\|_{L^2(I_j)}^2
 +C_{\eta,j}\|\phi\|_{L^2(I_j)}^2
\]
holds also on the one-sided interval ending at $R$; the Dirichlet trace
$\phi(R)=0$ causes no difficulty.  Since $a_-\in L^1(I_j)$, choose
$\eta$ so that the coefficient of the derivative term below is less
than a prescribed $\varepsilon>0$.  Equivalence of the weighted and
unweighted norms on $I_j$ then gives
\[
 \int_{I_j}a_-\phi^2r^{n-1}\,dr
 \leq \varepsilon\int_{I_j}|\phi'|^2r^{n-1}\,dr
 +C_{\varepsilon,j}\int_{I_j}\phi^2r^{n-1}\,dr.
\]
On the complement of these intervals the potential is bounded; it is
also bounded near the origin by assumption.  Summing the estimates
yields the asserted infinitesimal form bound.  Since the negative
part of the form is bounded by
 $\|a_+\|_\infty\|\phi\|_{L^2((0,R),r^{n-1}dr)}^2$, the preceding
 infinitesimal bound and the form perturbation theorem give closedness
 and lower boundedness.
Finally, compactness of the radial embedding
$H^1_{0,{\rm rad}}(B_R)\hookrightarrow L^2((0,R),r^{n-1}dr)$ gives
compact resolvent.
\end{proof}

\begin{lemma}\label{lem:radial-sturm}
{
Let $a=a_+-a_-\in L^1(0,R)$ satisfy the hypotheses of
Lemma~\ref{lem:singular-form}, and let $\mathcal L$ be the self-adjoint
operator in $L^2((0,R),r^{n-1}\,dr)$ associated with the resulting
closed, lower-bounded form on the regular radial Dirichlet domain. Its
differential expression is
\[
 {\tau y}:=-y''-\frac{n-1}{r}y'-a(r)y,
 \qquad y'(0)=0,
 \qquad y(R)=0.
\]
Let $y_0$ be the unique regular integral solution of
\begin{equation*}
 {\tau y_0=0},
 \qquad y_0(0)=1,
 \qquad y_0'(0)=0.
\end{equation*}
Assume that $y_0(R)\ne0$ and that $y_0$ has exactly $N$ zeros in
$(0,R)$.  Then $\mathcal L$ has exactly $N$ negative eigenvalues,
counted with multiplicity, and $\ker\mathcal L=\{0\}$.
}
\end{lemma}
\begin{proof}
{We approximate the singular coefficient by bounded ones, which avoids
introducing an artificial boundary condition near the origin.  Since
$a=a_+-a_-$ and $a_+\in L^\infty(0,R)$, define
\[
 a_m(r):=\max\{a(r),-m\}.
\]
For all sufficiently large $m$, $a_m=a$ in a fixed neighborhood of the
origin, while $a_m\in L^\infty(0,R)$ and
\[
 a_m\downarrow a,\qquad \|a_m-a\|_{L^1(0,R)}\longrightarrow0.
\]
Let $\mathcal L_m$ be the regular radial Dirichlet realization with
coefficient $a_m$, and let $y_m$ be the regular solution of
\[
 -y_m''-\frac{n-1}{r}y_m'-a_m(r)y_m=0,
 \qquad y_m(0)=1,\qquad y_m'(0)=0.
\]
For bounded $a_m$, the classical radial Sturm oscillation theorem gives that
the number of negative eigenvalues of $\mathcal L_m$ equals the number of
zeros of $y_m$ in $(0,R)$, provided $y_m(R)\ne0$; see, for example,
\cite{Titchmarsh}.

The $L^1$ continuous-dependence argument of Lemma~\ref{lem:log-L1}, applied
to the present linear equations, gives
\[
 y_m\longrightarrow y_0\qquad\text{in }C^1([0,R]).
\]
Every zero of $y_0$ in $(0,R)$ is simple, because a double zero would imply
$y_0\equiv0$ by linear Cauchy uniqueness.  Since $y_0(R)\ne0$, simple-zero
stability shows that, for all sufficiently large $m$, $y_m$ has exactly the
same $N$ zeros in $(0,R)$ and also satisfies $y_m(R)\ne0$.  Hence
$\mathcal L_m$ has exactly $N$ negative eigenvalues.

We next compare $\mathcal L_m$ with $\mathcal L$.  Let
\[
 \mathfrak q_m[\phi]
 :=\int_0^R\bigl(|\phi'|^2-a_m\phi^2\bigr)r^{n-1}\,dr,
 \qquad
 \mathfrak q[\phi]
 :=\int_0^R\bigl(|\phi'|^2-a\phi^2\bigr)r^{n-1}\,dr
\]
on the common form domain $H^1_{0,\mathrm{rad}}(B_R)$.  By
Lemma~\ref{lem:singular-form}, these forms are closed and have a common lower
bound.  Moreover,
\[
 0\le \mathfrak q[\phi]-\mathfrak q_m[\phi]
 =\int_0^R(a_m-a)\phi^2r^{n-1}\,dr.
\]
The difference $a_m-a$ is supported, for large $m$, only in fixed
neighborhoods of the finitely many singular points away from the origin.
On each such interval the weighted and unweighted $H^1$ norms are
equivalent, and the one-dimensional estimate
\[
 \|\phi\|_{L^\infty(I)}^2
 \le C_I\bigl(\|\phi'\|_{L^2(I)}^2+\|\phi\|_{L^2(I)}^2\bigr)
\]
therefore yields
\[
 0\le \mathfrak q[\phi]-\mathfrak q_m[\phi]
 \le \varepsilon_m
 \left(
 \int_0^R|\phi'|^2r^{n-1}\,dr+
 \int_0^R\phi^2r^{n-1}\,dr
 \right),
 \qquad \varepsilon_m\to0.
\]
{Since
$\mathfrak q[\phi]+(\|a_+\|_\infty+1)\|\phi\|_2^2$ controls the radial
$H^1$ norm, the preceding estimate is precisely form-norm convergence
$\mathfrak q_m\to\mathfrak q$.}  Since the radial form embedding is compact,
the min--max principle gives convergence of every ordered eigenvalue of
$\mathcal L_m$ to the corresponding eigenvalue of $\mathcal L$.

Finally, {standard one-dimensional regularity for the weak eigenvalue
equation with $L^1$ coefficient shows that} any element of
$\ker\mathcal L$ is a regular integral solution of the zero-energy equation.
Volterra uniqueness makes it a scalar multiple of $y_0$, and the Dirichlet
condition together with $y_0(R)\ne0$ forces that scalar to vanish.  Hence $\ker\mathcal L=\{0\}$.  Zero is therefore separated
from the discrete spectrum of $\mathcal L$, and the eigenvalue convergence
implies that $\mathcal L$ has the same number $N$ of negative eigenvalues as
$\mathcal L_m$ for all sufficiently large $m$.}
\end{proof}

\begin{proposition}\label{prop:zero-curve-principle}
{
Let $I_k=(a_k,\infty)$ and let $J_k\subset(0,\infty)$ be an open
interval.  Assume that:
\begin{enumerate}
\item {the $(k+1)$-st positive simple zero exists if and
only if $\alpha\in I_k$}; for every $\alpha\in I_k$, the first
$k+1$ positive zeros satisfy
$0<z_1(\alpha)<\cdots<z_{k+1}(\alpha)$ and are simple;
\item $z_{k+1}:I_k\to J_k$ is a $C^1$ bijection;
\item with
\[
 v(\cdot,\alpha):=\partial_\alpha u(\cdot,\alpha),
\]
the variation $v$ has exactly one zero in every interval
$(z_{j-1}(\alpha),z_j(\alpha))$, $1\le j\le k+1$, where
$z_0(\alpha):=0$, and
\begin{equation}\label{eq:positive}
 u_r(z_j(\alpha),\alpha)v(z_j(\alpha),\alpha)>0,
 \qquad 1\le j\le k+1;
\end{equation}
\item the nonlinear shooting problem and the regular linearized Cauchy
problem are unique up to the boundary zero, the latter in the integral
sense when its coefficient is singular;
\item {for every stopped profile, the coefficient
$a=f'(U)$ admits a decomposition $a=a_+-a_-$ satisfying the hypotheses of
Lemma~\ref{lem:singular-form}; in particular,
$a\in L^1(0,z_{k+1}(\alpha))$, $a$ is bounded near $r=0$, and the associated
radial quadratic form is closed and bounded from below.}
\end{enumerate}
Then, for every $R\in J_k$, there exists a unique regular radial
Dirichlet profile $U_{k,R}$ of \eqref{eq:unified-ball} that is positive
at the origin, has exactly $k$ simple zeros in $(0,R)$, and has a
simple boundary zero.  Moreover,
\[
 \ker_{\rm rad}L_{U_{k,R}}=\{0\},
 \qquad
 m_{\rm rad}(U_{k,R})=k+1.
\]
}
\end{proposition}

\begin{proof}
{
We divide the proof into two steps.

\medskip
\noindent\emph{Step 1: Existence and uniqueness.}
For $R\in J_k$, let $\beta\in I_k$ be the unique value such that
$z_{k+1}(\beta)=R$, and set
\[
 U_{k,R}:=u(\cdot,\beta)|_{[0,R]}.
\]
Assumption~1 shows that this stopped orbit has exactly $k$ simple
interior zeros and a simple boundary zero. Conversely, any profile in
the stated shooting class has a positive central value $\alpha\in I_k$
and, by nonlinear Cauchy uniqueness, equals $u(\cdot,\alpha)$ up to the
boundary. Hence $z_{k+1}(\alpha)=R$, and the injectivity of $z_{k+1}$
gives $\alpha=\beta$.

\medskip
\noindent\emph{Step 2: Radial nondegeneracy and radial Morse index.}
For each $R\in J_k$, Step~1 yields a unique $\beta_k(R)\in I_k$ such
that
\[
 z_{k+1}(\beta_k(R))=R.
\]
Let $\beta:=\beta_k(R)$, $U(r):=u(r,\beta)$, and $v(r):=\left.\partial_\alpha u(r,\alpha)\right|_{\alpha=\beta}$.
At the boundary, \eqref{eq:positive} with $j=k+1$ gives $U_r(R)v(R)>0$,
and therefore $v(R)\ne0$.

Differentiating the shooting equation
\eqref{eq:intro-shooting-unified} with respect to $\alpha$ shows that
$v$ is the regular integral solution of the zero-energy linearized
problem
\[
 L_Uv=0,
 \qquad v(0)=1,
 \qquad v'(0)=0,
\]
where
\[
 L_U=-\frac{d^2}{dr^2}-\frac{n-1}{r}\frac{d}{dr}-a(r),
 \qquad a(r)=f'(U(r)).
\]
By the linear Cauchy uniqueness in Assumption~4, $v$ is precisely the
regular zero-energy solution $y_0$ appearing in
Lemma~\ref{lem:radial-sturm}. Assumption~3 gives exactly one zero of
$v$ in each of
\[
 (0,z_1),\ (z_1,z_2),\ \ldots,\ (z_k,R),
\]
so $v$ has exactly $k+1$ zeros in $(0,R)$. {Assumption~5 places $a$ under Lemmas~\ref{lem:singular-form}
and~\ref{lem:radial-sturm}.} Hence all the hypotheses of
Lemma~\ref{lem:radial-sturm} hold with $y_0=v$ and $N=k+1$. Consequently,
\[
 \ker_{\rm rad}L_U=\{0\},
 \qquad
 m_{\rm rad}(U)=k+1.
\]
}
\end{proof}

\subsection{The power equation}
Let $f=f_{\rm pow}$. The nonlinear shooting problem and its
linearization have the standard Cauchy theory. {Tang's zero-curve analysis} \cite{Tang}
for $n\geq3$, together with the planar result of
Zhang--Zhang \cite{ZhangZhang}, shows that, for every {$k\geq0$},
\[
 z_{k+1}:(\alpha_k^{\rm pow},\infty)\longrightarrow(0,\infty)
\]
is a strictly decreasing $C^1$ bijection. Their phase-transition
analysis also gives the required zero distribution of $v=\partial_\alpha u$ and the sign relation
$u_r(z_j(\alpha),\alpha)
v(z_j(\alpha),\alpha)>0$.  Moreover, $f_{\rm pow}'(U^{\rm pow}_{k,R})$ is bounded on $[0,R]$,
so the associated radial quadratic form is the standard closed,
lower-bounded Dirichlet form. All the hypotheses of
Proposition~\ref{prop:zero-curve-principle} are therefore satisfied.
Consequently,
\[
 \ker_{\rm rad}L_{U^{\rm pow}_{k,R}}=\{0\},
 \qquad
 m_{\rm rad}(U^{\rm pow}_{k,R})=k+1.
\]
This proves Theorem~\ref{thm:morse-main}(a).

\subsection{The logarithmic equation}
Let $U=U^{\log}_{k,R}$.  Theorems~\ref{thm:log-main} and~\ref{thm:log-whole},
together with Proposition~\ref{prop:log-planar}, provide the required
zero-curve bijection, zero-sign relation, and interlacing property of the
shooting variation.  At every simple
zero $z$ of $U$,
\[
 \log U^2(r)+2=2\log|r-z|+O(1)\quad\hbox{as }r\to z.
\]
Since the nodal set of $U$ is finite, it follows that
$\log U^2+2\in L^1(0,R)$. Lemma~\ref{lem:log-L1} therefore gives uniqueness for the regular
linearized Cauchy problem across the nodal set.
{Moreover, $a=\log U^2+2$ is bounded above, whereas
$(-a)_+$ has only finitely many logarithmic $L^1$ singularities and
is bounded near the origin. Lemma~\ref{lem:singular-form} therefore
shows that the radial quadratic form is closed, bounded from below,
and has compact resolvent.} Thus all the hypotheses of
Proposition~\ref{prop:zero-curve-principle} are satisfied.
Consequently,
\[
 \ker_{\rm rad}L_U=\{0\},
 \qquad
 m_{\rm rad}(U)=k+1.
\]
This proves Theorem~\ref{thm:morse-main}(b).

\subsection{The sublinear equation}

Let $U=U_{k,R}^{\rm sub}$.  Near every simple zero $z$ of $U$,
\[
 {|U(r)|^{q-2}\asymp |r-z|^{q-2}\in L^1_{\rm loc}},
\]
since $q-2>-1$. Thus the linearized equation is well defined across
the nodal set in the integral sense, and the corresponding linear
Cauchy problem is unique. {Here
\[
 \mathcal Q_U(\phi)=\int_0^R
 \bigl(|\phi'|^2-\phi^2+(q-1)|U|^{q-2}\phi^2\bigr)r^{n-1}\,dr.
\]
The singular term is nonnegative and has only finitely many
$L^1$ singularities of the form $|r-z|^{q-2}$. The only negative
potential is the bounded term $-1$. Lemma~\ref{lem:singular-form}
therefore implies that $\mathcal Q_U$ is closed, bounded from below,
and has compact resolvent.}
The zero-curve bijection, the zero-sign relation, and the required
interlacing of the shooting variation follow from the results
established in the sublinear section, in particular
Lemma~\ref{lem:sublinear-zero-curve} and the corresponding
phase-transition results. Therefore all the hypotheses of
Proposition~\ref{prop:zero-curve-principle} are satisfied, and
\[
 \ker_{\rm rad}L_U=\{0\},\qquad
 m_{\rm rad}(U)=k+1.
\]
This proves Theorem~\ref{thm:morse-main}(c). Together with the power and logarithmic cases proved above, this
proves the radial assertions of Theorem~\ref{thm:morse-main}.

\subsection{Angular decomposition and the full Morse index}

{We now pass from the radial index to the full Dirichlet spectrum.
For each degree $\ell\ge0$, let $Y_\ell$ be a spherical harmonic satisfying}
\[
 -\Delta_{\mathbb S^{n-1}}Y_\ell
 =\ell(\ell+n-2)Y_\ell.
\]
On functions $\phi(x)=\psi(r)Y_\ell(x/r)$, the quadratic form is
\begin{equation}\label{eq:angular-form}
 \mathcal Q_{U,\ell}(\psi)
 =\int_0^R\left(
 |\psi'|^2+\frac{\ell(\ell+n-2)}{r^2}\psi^2
 -f'(U)\psi^2\right)r^{n-1}\,dr.
\end{equation}
{The proof of Lemma~\ref{lem:singular-form} applies verbatim in every
fixed angular sector, because the singular nodal radii stay away from the
origin and the centrifugal term is nonnegative.  Since $f'(U)$ is bounded
above in all three models, the sector forms have a common lower bound; their
orthogonal direct sum therefore defines the closed full Dirichlet form used
in Theorem~\ref{thm:morse-main}.}
For
$\ell\ge1$, writing $\psi=r^\ell\eta$ and integrating the cross term gives
the exact identity
\begin{equation*}
 \mathcal Q_{U,\ell}(r^\ell\eta)
 =\int_0^R\bigl(|\eta'|^2-f'(U)\eta^2\bigr)
 r^{n+2\ell-1}\,dr.
\end{equation*}
Thus every angular sector is a regular radial Sturm problem in the effective
dimension $n+2\ell$; in particular, its spectrum is discrete.

The case $\ell=0$ has already been proved and gives $N_0(U)=k+1$ with
trivial kernel. {For $\ell=1$, differentiation of the radial equation gives, in
the integral sense across the singular nodal radii in the logarithmic and
sublinear cases,
\begin{equation*}
 -(U')''-\frac{n-1}{r}(U')'
 +\frac{n-1}{r^2}U'-f'(U)U'=0.
\end{equation*}
We first count the zeros of $U'$ directly.  For each of the three
nonlinearities, if $F(s):=\int_0^s f(t)\,dt$ and $F(s)>0$, then
$sf(s)>0$.  The radial energy
$E=\frac12U'^2+F(U)$ is strictly decreasing along every nonconstant nodal
arc.  At an interior critical point, comparison with the following simple
zero gives $F(U)=E>0$.  Therefore $U''=-f(U)$ has the sign opposite to $U$:
every critical point in a positive nodal domain is a strict local maximum,
and every critical point in a negative nodal domain is a strict local
minimum.

At the first zero $z_1$, strict energy dissipation gives
$F(U(0))=E(0)>E(z_1)>0$.  Hence $f(U(0))>0$ and
$U''(0)=-f(U(0))/n<0$, so $U'$ is negative immediately to the right of
the origin.  If $r_*$ were its first zero in the positive first nodal domain,
then $U'<0$ on $(0,r_*)$ would force $U''(r_*)\ge0$, whereas the preceding
critical-point sign gives $U''(r_*)<0$.  Thus there is no interior critical
point before $z_1$.  Rolle's theorem supplies a critical point between every
two consecutive zeros.  Two such points in the same nodal domain would force
an intervening critical point of the opposite type, which the sign rule
excludes; hence the critical point is unique.  Consequently, $U'$ has exactly
$k$ zeros in $(0,R)$.

The same calculation shows $U''(0)\ne0$, so $U'(r)/r$ extends to a nonzero
regular solution at the origin; moreover $U'(R)\ne0$ because the boundary
zero is simple.  Applying Lemma~\ref{lem:radial-sturm} to $U'/r$ in the
effective dimension $n+2$ therefore gives
\[
 N_1(U)=k,
 \qquad \ker L_{U,1}=\{0\}.
\]}

{For $\ell\ge1$, the sector form domains coincide.  Hence, for
$\ell\ge2$, the min--max principle applied to \eqref{eq:angular-form} shows
that every ordered sector eigenvalue is nondecreasing in $\ell$.  Therefore
$0\le N_\ell(U)\le N_1(U)=k$.}
Furthermore, $f'(U)$ is bounded above in all three models. Since
$r^{-2}\ge R^{-2}$,
\[
 \mathcal Q_{U,\ell}(\psi)
 \ge\left(\frac{\ell(\ell+n-2)}{R^2}
 -\sup_{0<r<R}f'(U(r))\right)
 \int_0^R\psi^2r^{n-1}\,dr.
\]
{For all sufficiently large $\ell$, the coefficient in parentheses is
strictly positive; hence $N_\ell(U)=0$ and the sector kernel is trivial.
Finally, orthogonal spherical-harmonic decomposition of $H_0^1(B_R)$ and the
multiplicity formula for degree $\ell$ give
\eqref{eq:main-full-morse}.} The radial and degree-one kernel
statements proved above localize every possible full-space degeneracy to one
of the finitely many degrees $\ell\ge2$. This completes the proof of
Theorem~\ref{thm:morse-main}.

\vspace{.3cm}

\noindent\textbf{Acknowledgments.}
This work was supported by the Chongqing Natural Science Foundation, China
(CSTB2024NSCQ-LZX0038).

\medskip

\noindent\textbf{Conflict of interest.}
The authors declare that they have no conflict of interest.

\medskip
\noindent\textbf{Data availability.}
Data sharing is not applicable to this article as no data were created or analyzed in this study.

\end{document}